\documentclass[a4paper,11pt]{amsart}
\usepackage{geometry}
\usepackage[british]{babel}
\usepackage{ae}
\usepackage{amsfonts}
\usepackage{amssymb}
\usepackage{amsthm}
\usepackage{graphicx}
\usepackage{hyperref}
\usepackage{multirow}
\usepackage{color}
\usepackage{youngtab}
\usepackage{young}
\usepackage{amsmath}
\usepackage{amscd}

\newcounter{itemcounter}
\numberwithin{itemcounter}{subsection}

\newtheorem{thm}[itemcounter]{Theorem}
\newtheorem{lem}[itemcounter]{Lemma}
\newtheorem{prop}[itemcounter]{Proposition}
\newtheorem{cor}[itemcounter]{Corollary}

\newtheorem{con}[itemcounter]{Conjecture}

\newtheorem{rem}[itemcounter]{Remark}

\title{On Rouquier Blocks for Finite Classical Groups at Linear Primes}
\author{Michael Livesey}

\begin{document}

\begin{abstract}
H. Miyachi and W. Turner have independently proved that Brou\'{e}'s Abelian
Defect Group Conjecture holds for
certain unipotent blocks of the finite general linear group, the so-called
Rouquier blocks\cite{miyach2001} and\cite[Section 2, Theorem 1]{turner2002}.
This together with A.
Marcus\cite[Theorem 4.3(b)]{marcus1994} and J. Chuang and R.
Rouquier\cite[Theorem 7.18]{churou2008} proves that the conjecture holds for all
blocks of such groups. We prove that other finite classical groups also possess
unipotent Rouquier blocks at linear primes.
\end{abstract}

\maketitle

\tableofcontents
\newpage

\section{Introduction and Notation}\label{intro}

Let $p$ be a prime and consider the $p$-modular system $(K,\mathcal{O},k)$ such
that $K$ contains enough roots of unity for the groups being considered in this
paper.
\newline
\newline
Let $G$ be a finite group and $b$ a block idempotent of $\mathcal{O}G$ with
defect group $P$. Recall that to each subgroup $H$ of $G$ containing $N_G(P)$
there exists a unique block $\mathcal{O}Hc$ associated to $b$ through the Brauer
homomorphism.
\newline
\newline
We state Brou\'{e}'s Abelian Defect Group Conjecture\cite[Chapter
6.3.3]{konzim1998}.

\begin{con}[Brou\'{e}]
Let $G$ be a finite group and $P$ an abelian $p$-subgroup. Let $b$ be a block
idempotent of $\mathcal{O}G$ with defect group $P$ and Brauer correspondent $c$
in $N_G(P)$. Then $\mathcal{O}Gb$ and $\mathcal{O}N_G(P)c$ are derived
equivalent.
\end{con}

The conjecture is known to hold for symmetric groups. Recall that to each block
$B$ of $\mathcal{O}S_n$ there is an associated non-negative integer $w$ called
the weight of $B$. The defect group $P$ of $B$ is abelian if and only if $w<p$,
in which case $P\cong(C_p)^w$. The proof consists of three steps:

\begin{enumerate}
\item For each $w$ with $(0\leq w<p)$ there exists an $n\geq1$ and a block $B$
of $\mathcal{O}S_n$ which is Morita equivalent to the principal block of
$\mathcal{O}(S_p\wr S_w)$\cite[Section 3, Theorem 2]{chukes2002}.
\item The principal block of $\mathcal{O}(S_p\wr S_w)$ is derived equivalent to
the Brauer correspondent of $B$ in $N_{S_n}(P)$\cite[Theorem
4.3(b)]{marcus1994}.
\item Any two blocks $B$, $B'$ of $\mathcal{O}S_n$, $\mathcal{O}S_{n'}$
respectively of the same weight have Morita equivalent Brauer correspondents.
\item Any two blocks $B$, $B'$ of $\mathcal{O}S_n$, $\mathcal{O}S_{n'}$
respectively of the same weight are derived equivalent\cite[Theorem
7.2]{churou2008}.
\end{enumerate}

These methods have been adapted for unipotent blocks of finite general linear
groups\cite[Section 2, Theorem 1]{turner2002}, \cite[Theorem 7.18]{churou2008}.
We investigate (1), (2) and (3) for unipotent blocks of other finite classical
groups.
\newline
\newline
Let $q$ be a prime power (we allow $q$ to be even only in the case of the unitary group)
and $G=G_m(q)$ be a group of the form $U_n(q)$, $Sp_{2n}(q)$,
$CSp_{2n}(q)$, $SO_{2n+1}(q)$, $SO^+_{2n}(q)$, $SO^-_{2n}(q)$, $CSO^+_{2n}(q)$
or $CSO^-_{2n}(q)$. We adopt the notation that elements of $U_n(q)$ have
entries in $\mathbb{F}_q$ so $q$ is a square, say $q=q_0^2$. For a positive
integer $d$ we let $J_d$ be the $d\times d$ matrix with entries in
$\mathbb{F}_q$ with $1$s along the antidiagonal and $0$s elsewhere. To each of
the above groups types we
associate the $2$-fold extension $GL_d(q).2=\langle GL_d(q),s\rangle$ of
$GL_d(q)$ described as follows:

\begin{enumerate}
 \item[(i)] \[ s^2= \bigg\{
   \begin{matrix}
     -1 & \text{for $Sp_{2n}(q)$ or $CSp_{2n}(q)$}\\
     1 & \text{otherwise}
   \end{matrix}
\]
 \item[(ii)] and for all $A\in GL_d(q)$
\[ sAs^{-1}= \Bigg\{
   \begin{matrix}
     J_dA^{-q_0t}J_d & \text{for $U_n(q)$}\\
     -J_dA^{-t}J_d & \text{for $Sp_{2n}(q)$ or $CSp_{2n}(q)$}\\
     J_dA^{-t}J_d & \text{otherwise}
   \end{matrix}
\]
\end{enumerate}

We have a natural homomorphism $GL_d(q).2\rightarrow\{\pm1\}$
with kernel $GL_d(q)$. This extends to a map $GL_d(q).2\wr
S_w\rightarrow\{\pm1\}$ and we use $M$ to denote the kernel of this map.
\newline
\newline
We now recall some facts about the unipotent blocks of $\mathcal{O}G$, more
details can be found in section~\ref{sec:blocks}. To each
unipotent block $B$ of $\mathcal{O}G$ there is an associated non-negative
integer $w$ called the weight of $B$. As with the symmetric group the defect
group of $B$ is abelian if and only if $w<p$. In the case of $SO^+_{2n}(q)$ or
$CSO^+_{2n}(q)$ we have the notion of $B$ being degenerate. For a fixed prime
power $q$ we have the notion of $p$ being a linear or unitary prime with
respect to $q$.
\newline
\newline
Now let $p$ be a linear prime with respect to $q$ and $d$ the multiplicative
order of $q\mod{p}$.
\newline
\newline
We now state our main theorem.

\begin{thm}\label{thm:intro}(cf. Theorem~\ref{thm:nondegen},
Theorem~\ref{thm:degen}, Corollary~\ref{cor:so2n}, Lemma~\ref{lem:twist})
Let $p$ be a linear prime with respect to $q$. For $CSp_{2n}(q)$ and
$CSO^\pm_{2n}(q)$ let's denote
by $Z_p$ the $p$-part of the the multiplicative group
$\mathbb{F}_q^\times$. For all other groups $Z_p$ will be the trivial group. For
all $w$ $(0\leq w<p)$ there
exists some $m=m(w)$ and some unipotent block $B$ of $\mathcal{O}G_m(q)$ of
weight $w$ such that:
\begin{enumerate}
 \item $B$ is Morita equivalent to the principal block of
$\mathcal{O}((GL_d.2\wr S_w)\times Z_p)$ if $G=U_n(q)$, $Sp_{2n}(q)$,
$CSp_{2n}(q)$ or $SO_{2n+1}(q)$ or if $G=SO^\pm_{2n}(q)$ or $CSO^\pm_{2n}(q)$
and $B$ is
non-degenerate.
 \item $B$ is Morita equivalent to the principal block of $\mathcal{O}(M\times
Z_p)$ if $G=SO^+_{2n}(q)$ or $CSO^+_{2n}(q)$ and $B$ is degenerate.
\end{enumerate}
\end{thm}

We will then prove that the Brauer correspondent of $B$ in $N_G(P)$ is derived
equivalent to the appropriate block in the above theorem and thus we obtain:

\begin{cor}(cf. Corollary~\ref{cor:bottom})
The block $B$ of $\mathcal{O}G$ satisfies Brou\'{e}'s Abelian Defect Group
Conjecture.
\end{cor}

Finally we will prove that if $B$ and $B'$ are unipotent blocks of
$\mathcal{O}G_m(q)$ and
$\mathcal{O}G_{m'}(q)$ respectively with the same weight and either both
non-degenerate or both degenerate then they have Morita equivalent Brauer
correspondents (see Corollary~\ref{cor:norm}). We will then have an analogue for parts (1), (2) and (3) in the
proof of the conjecture for the symmetric groups for unipotent blocks at linear
primes of our finite classical groups.
\newline
\newline
We should note that the Morita equivalences we construct in~\ref{thm:nondegen},~\ref{thm:degen}
and~\ref{cor:so2n} are splendid in the sense of\cite[Chapter 9.2.5]{konzim1998}
and that they ultimately gives rise to splendid derived equivalences between $B$
and its Brauer correspondent in $N_G(P)$.
\newline
\newline
Sections 2,3 and 4 give some basic background information before our finite
classical groups are introduced. In section 2 we describe the combinatorial
objects partitions and symbols. They will later be used to label the characters
and blocks of our finite classical groups. In section 3 we will take a brief
look at the Weyl groups of type $A_n$, $B_n$ and $D_n$ including branching rules
which will ultimately be used to describe what happens to characters of our
finite classical groups under Harish-Chandra induction. Section 4 consists of a
brief overview of the Brauer homomorphism and the Brauer correspondence.
\newline
\newline
Section 5 is where we first introduce the finite classical groups that are the
subject of this paper. We do this in two ways. First as the group of fixed
points of some algebraic group under some Frobenius endomorphism and secondly as
the group of linear maps that preserve or scale some bilinear form. We then go on
to describe certain Levi subgroups of the finite classical groups. In section 6
we set up a labeling for the characters of the groups in question. Section 7
describes Harish-Chandra induction from the Levi subgroups described in section
5 and we look at what effect this has on characters. Unipotent blocks are introduced
in section 8 and we describe the defect groups as well as exactly what
characters are in such a block.
\newline
\newline
Our main theorems~\ref{thm:nondegen} and~\ref{thm:degen} are proved in section
9 and we then go on to look at how the main theorems are applied to Brou\'{e}'s Abelian
Defect Group Conjecture. In section 10 we go on to prove the analogues of steps (2) and (3) in the
proof of Brou\'{e}'s Abelian Defect Group Conjecture for the symmetric groups.

\section{Some Combinatorics}
\subsection{Partitions and Symbols}
For our description of partitions and symbols we follow G. Hiss and R.
Kessar\cite[Section 2.2]{hiskes2000}.
\newline
\newline
A partition $\lambda$ of a non-negative integer $n$ is a finite ordered set
$(\alpha_1,\alpha_2,\dots,\alpha_t)$ of positive integers with
$\alpha_1\geq\alpha_2\geq\dots\geq\alpha_t$ and
$\alpha_1+\alpha_2+\dots+\alpha_t=n$. We write $\lambda\vdash n$ or
$|\lambda|=n$.
\newline
A beta-set of a partition $\lambda=(\alpha_1,\alpha_2,\dots,\alpha_t)$ is a
finite set of non-negative integers $\{\beta_1,\beta_2,\dots,\beta_s\}$ where
$\beta_1<\beta_2<\dots<\beta_s$ and $\alpha_i=\beta_{s-i+1}-s+i$ for $1\leq
i\leq s$ and $\alpha_i$ is considered to be zero for $i>t$.
\newline
Now if $d$ is a non-negative integer then the $d$-shift of a beta-set
$\{\beta_1,\beta_2,\dots,\beta_s\}$ is
$\{0,1,\dots,d-1\}\cup\{\beta_1+d,\beta_2+d,\dots,\beta_s+d\}$.
\newline
Two beta-sets are said to be equivalent if one is the $d$-shift of the other for
some non-negative integer $d$ and two beta-sets give rise to the same partition
if and only if they are equivalent.
\newline
\newline
A symbol is an unordered pair of beta-sets $\{X,Y\}$. A symbol $\{X,Y\}$ is said
to be degenerate if $X=Y$. Two symbols are said to be equivalent if one of them
can be obtained from the other by simultaneous $d$-shifts on both parts of the
symbol for some non-negative integer $d$.
\newline
\newline
The defect of a symbol $\{X,Y\}$ is the quantity:
\begin{center}
$||X|-|Y||$
\end{center}

The rank is:

\begin{center}
$\sum_{x\in X}+\sum_{y\in Y}-\lfloor(\frac{|X|+|Y|-1}{2})^2\rfloor$
\end{center}

Both defect and rank are constant on equivalence classes of symbols.

\subsection{Hooks and Cores}\label{sec:hook}

The Young diagram for a partition $\lambda=(\alpha_1,\alpha_2,\dots,\alpha_t)$
is an arrangement of boxes with $\alpha_i$ boxes in the $i$th row. For example
the partition $(5,5,3,2)$ has Young diagram:

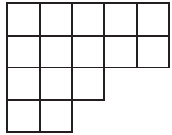
\begin{center}
$\yng(5,5,3,2)$
\end{center}

A hook of a partition is any box in its Young diagram together with everything directly
below it and directly to the right of it in the diagram. The length of a hook is
equal to the number of boxes in the hook. If a hook is of length $e$ it is said
to be an $e$-hook.
\newline
The partition corresponding to the diagram obtained by deleting a hook and
moving everything below and to the right of the hook one box up and one box left
is said to be obtained by removing this hook. For example:

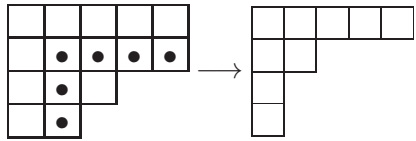
\begin{center}
$\begin{CD}
\begin{Young}
&&&&\cr
&$\bullet$&$\bullet$&$\bullet$&$\bullet$\cr
&$\bullet$&\cr
&$\bullet$\cr
\end{Young}\end{CD}\longrightarrow
\begin{CD}
$\yng(5,2,1,1)$
\end{CD}$
\end{center}

shows a $6$-hook and how to delete it.
\newline
\newline
If $\beta$ is a corresponding beta-set then an $e$-hook corresponds to a pair of
non-negative integers $(x,y)$ such that $x\in\beta$ and $y\not\in\beta$ with
$x-y=e$. Removing this hook corresponds to replacing $x$ with $y$ in $\beta$.
\newline
Similarly if $\beta$ is a beta-set and $(x,y)$ a pair of non-negative integers
with $x\in\beta$, $y\not\in\beta$ and $y-x=e$ then the beta-set obtained
obtained by replacing $x$ with $y$ is said to obtained by adding an $e$-hook to
$\beta$. The corresponding transformation to the partition is also called adding
an $e$-hook.
\newline
\newline
For the rest of the section let's fix a positive integer $e$. Beta-sets can be
displayed on an $e$-abacus as follows:
\newline
Take an abacus with $e$ columns (runners) labeled from left to right by
$0,1,\dots,e-1$ and with the rows labeled from the top by $0,1,2,\dots$. From
now on this will be known as an $e$-abacus. Now given some beta-set $\beta$ we
can represent it on the abacus by putting a bead on the $j$th runner and $i$th
row if and only if $ei+j\in\beta$.
\newline
Removing/adding an $e$-hook corresponds to moving a bead up/down a runner one
place into an unoccupied position. From this description it is clear that
removing $e$-hooks until there are no more $e$-hooks to remove gives a
well-defined beta-set called the $e$-core of $\beta$. The corresponding
partition is called the $e$-core of the partition corresponding to $\beta$.
Again this is well-defined, in other words equivalent beta-sets have equivalent cores. If
$\lambda$ is a partition and $\lambda'$ its $e$-core then the $e$-weight of
$\lambda$ is defined to be $\frac{|\lambda|-|\lambda'|}{e}$. In other words the
weight is defined to be the number of $e$-hooks removed to obtain the $e$-core.
\newline
\newline
Given a beta-set $\beta$ we can construct a beta-set $\beta^j$ by letting
$i\in\beta_j$ if and only if $ei+j\in\beta$. These beta-sets are called the
$e$-quotients of $\beta$. Then a beta-set is completely defined by the
$\beta^j$s $(0\leq j\leq e-1)$. The $e$-quotients are written as
$[\beta^0,\beta^1,\dots,\beta^{e-1}]$. If $\lambda$ is the partition
corresponding to $\beta$ then we say $[\lambda^0,\lambda^1,\dots,\lambda^{e-1}]$
are the $e$-quotients of $\lambda$ where $\lambda^i$ is the partition
corresponding to $\beta^i$. This is uniquely determined by $\lambda$ up to a cyclic
permutation of the $\lambda^i$s.
\newline
\newline
Given a symbol $\{X,Y\}$ we again have a notion of hook. An $e$-hook is a pair
$(x,y)$ of non-negative integers with $x\in X$ but $y\not\in X$ (or $X$ replaced
with $Y$) with $x-y=e$. Removing this $e$-hook means replacing $x$ with $y$ in
$X$ (or $Y$). There is also a completely analogous notion of adding hooks.
\newline
\newline
Now consider the $2e$-abacus. For integers $i$ and $j$ with $0\leq i$ and $0\leq
j\leq e-1$ put a bead on the $i$th row and $j$th runner if and only if $ei+j\in
X$ and on the $i$th row and $(e+j)$th runner if and only if $ei+j\in Y$. This is
called the $2e$-linear diagram of $\{X,Y\}$.
\newline
\newline
This diagram shows that removing $e$-hooks until no more can be removed is a
well-defined process producing what is called the $e$-core of $\{X,Y\}$. Also
equivalent symbols have equivalent $e$-cores. If when we remove all $e$-hooks
from a symbol we get a degenerate symbol and this process involved removing a
positive number of $e$-hooks we say that the $e$-core is $2$ copies of this
degenerate symbol. However, if this process involved removing no $e$-hooks then
we say the $e$-core is just $1$ copy of this degenerate symbol. From now on we
will use symbol to mean an equivalence class of symbols.

\subsection{Alternative Description}\label{not:char}

When we are looking at characters of our finite classical groups we will need
the following alternative description of partitions and symbols as given
in\cite[Section 5.2]{hiskes2000}.
\newline
\newline
Given a partition we display it on a $2$-abacus making sure the $0$th runner has
more beads than the $1$th. Say it has $s$ more beads. Then we relabel the
partition $(s,\mu,\nu)$ where $[\mu,\nu]$ are the $2$-quotient partitions with
respect to the abacus representation given.
\newline
Given a symbol $\{X,Y\}$ with $|X|>|Y|$ we relabel it $(s,\mu,\nu)$ where
$s=|X|-|Y|$, $\mu$ is the partition corresponding to $X$ and $\nu$ is the
partition corresponding to $Y$. If we have a symbol $\{X,Y\}$ with $|X|=|Y|$ we
relabel it $(0,\mu,\nu)$ where $\mu$ is the partition corresponding to $X$ and
$\nu$ is the partition corresponding to $Y$. Note that in this final case $(0,\mu,\nu)$ and $(0,\nu,\mu)$
correspond to the same symbol.

\subsection{Littlewood-Richardson Coefficients}\label{sec:LitRic}

Let $\mu$, $\nu$ and $\lambda$ be partitions such that $|\mu|=m$, $|\nu|=n$ and
$|\lambda|=(n+m)$. The Littlewood Richardson coefficient $g^\lambda_{\mu,\nu}$
is described as follows\cite[Definition 16.1]{james1978}:
\newline
Let $\mu=(\alpha_1,\alpha_2,\dots,\alpha_s)$,
$\nu=(\beta_1,\beta_2,\dots,\beta_t)$ and
$\lambda=(\gamma_1,\gamma_2,\dots,\gamma_u)$. If $s>u$ or $\alpha_i>\gamma_i$
for some $i$ with $0\leq i\leq s$ then $g^\lambda_{\mu,\nu}$ is zero.
\newline
Otherwise lie the Young diagram of $\mu$ on top of that of $\lambda$ so that the
top left boxes coincide. Now $g^\lambda_{\mu,\nu}$ is the number of ways the
complement of the Young diagram of $\mu$ inside that of $\lambda$ can be filled
with positive integers according to the following rules.

\begin{enumerate}
\item There must appear $\beta_i$ $i$'s.
\item The integers must be non-decreasing from left to right along rows and
strictly increasing down columns.
\item Reading each row from right to left starting with the top row and
continuing downwards gives a sequence of integers such that at no point does the
number of $(i+1)$s exceed the number of $i$s for all positive integers $i$.
\end{enumerate}

\section{Weyl Groups}\label{sec:weyl}

For the branching rules of the following Weyl groups we follow G. Hiss and K.
Kessar\cite[Section 3]{hiskes2000}. We will have repeated use of the
Littlewood-Richardson coefficients~\ref{sec:LitRic}.

\begin{description}
\item[$\mathbf{A_n}$] The Weyl group of type $A_n$, the symmetric group
$S_{n+1}$ on $(n+1)$ letters, has presentation:

\begin{center}
$S_{n+1}=\langle s_1,\dots,s_n|s_i^2=1, (s_is_{i+1})^3=1,(s_is_j)^2=1$ $|i-j|>1
\rangle$
\end{center}

The ordinary characters of $S_n$ are labeled by partitions of $n$. If
$\alpha\vdash n$ we use $\chi^\alpha$ to denote the corresponding character of
$S_n$.
\newline
\newline
We are concerned with inducing characters from $S_{n-k}\times S_k$, generated
by:

\begin{align*}
\{s_1,\dots,s_{n-k-1},s_{n-k+1},\dots,s_{n-1}\}
\end{align*}

to $S_n$.
\newline
\newline
If $\alpha\vdash (n-k)$, $\beta\vdash k$ and $\gamma\vdash n$ then the
multiplicity of $\chi^\gamma$ in $Ind_{S_{n-k}\times
S_k}^{S_n}(\chi^\alpha\otimes\chi^\beta)$ is $g_{\alpha,\beta}^\gamma$.

\item[$\mathbf{B_n}$] The Weyl group of type $B_n$, denoted $W_n$, has
presentation:
\begin{center}
$W_n=\langle s_1,\dots,s_n|s_i^2=1, (s_1s_2)^4=1, (s_is_{i+1})^3=1$ $i>1,
(s_is_j)^2=1$ $|i-j|>1 \rangle$
\end{center}

The ordinary characters of $W_n$ are labeled by bi-partitions of $n$, a
bi-partition of $n$ is an ordered pair of partitions $(\alpha^0,\alpha^1)$ such
that $|\alpha^0|+|\alpha^1|=n$ and we write $(\alpha^0,\alpha^1)\vdash n$. If
$(\alpha^0,\alpha^1)\vdash n$ we use $\chi^\alpha$ to denote the corresponding
character of $W_n$.
\newline
\newline
We are concerned with inducing characters from $W_{n-k}\times S_k$, generated
by:

\begin{align*}
\{s_1,..,s_{n-k},s_{n-k+2},\dots,s_n\}
\end{align*}

to $W_n$.
\newline
\newline
If $(\alpha^0,\alpha^1)\vdash (n-k)$, $\gamma\vdash k$ and
$(\beta^0,\beta^1)\vdash n$ then set $j=|\beta^0|-|\alpha^0|$. The multiplicity
of $\chi^\beta$ in $Ind_{W_{n-k}\times
S_k}^{W_n}(\chi^\alpha\otimes\chi^\gamma)$ is zero if $j<0$ or $j>k$. Otherwise
it is:

\begin{center}
$\sum_{\delta^0\vdash j}\sum_{\delta^1\vdash
k-j}g_{\alpha^0,\delta^0}^{\beta^0}g_{\alpha^1,\delta^1}^{\beta^1}g_{\delta^0,
\delta^1}^\gamma$
\end{center}

\item[$\mathbf{D_n}$] The Weyl group of type $D_n$, denoted $\widetilde{W}_n$,
has presentation:
\begin{center}
$\widetilde{W}_n=\langle s_1,\dots,s_n|s_i^2=1, (s_1s_2)^2=1, (s_1s_3)^3=1,
(s_is_{i+1})^3=1$ $i>1, (s_is_j)^2=1$ $|i-j|>1$ $\{i,j\}\neq\{1,3\} \rangle$
\end{center}

Note $\widetilde{W}_n$ can be viewed as a subgroup of $W_n$ of index 2 with
generators $\{s_1s_2s_1,s_2,\dots,s_n\}$.
\newline
\newline
If $(\alpha^0,\alpha^1)\vdash n$ with $\alpha^0\neq\alpha^1$ then the character
$\chi^{(\alpha^0,\alpha^1)}$ of $W_n$ restricts to a single character of
$\widetilde{W}_n$ and $\chi^{(\alpha^1,\alpha^0)}$ restricts to the same
character. We denote this character $\chi^{\{\alpha^0,\alpha^1\}}$ and describe
it as non-degenerate. However, if
$\alpha^0=\alpha^1$ then $\chi^{(\alpha^0,\alpha^0)}$ restricts to the sum of
$2$ distinct characters of $\widetilde{W}_n$. We denote the $2$ characters
$\chi^{\{\alpha^0,\alpha^0\}}$ and $\chi'^{\{\alpha^0,\alpha^0\}}$ and describe
them as degenerate. Now
conjugation by any element $g\in W_n\backslash\widetilde{W}_n$ induces an
automorphism on $\widetilde{W}_n$ that swaps $\chi^{\{\alpha^0,\alpha^0\}}$ and
$\chi'^{\{\alpha^0,\alpha^0\}}$.
\newline
\newline
We are concerned with inducing characters from $\widetilde{W}_{n-k}\times S_k$,
generated by:

\begin{align*}
\{s_1,\dots,s_{n-k},s_{n-k+2},\dots,s_n\}
\end{align*}
to $\widetilde{W}_n$.
\newline
\newline
Given the information above we can deduce the multiplicity of an irreducible
character of $\widetilde{W}_n$ in that of an irreducible character of
$\widetilde{W}_{n-k}\times S_k$ induced up to $\widetilde{W}_n$. We consider
three cases. When the characters of $\widetilde{W}_{n-k}$ and $\widetilde{W}_n$
are both non-degenerate, when the character of $\widetilde{W}_{n-k}$ is
non-degenerate but that of $\widetilde{W}_n$ is degenerate and when the
character of $\widetilde{W}_{n-k}$ is degenerate but that of $\widetilde{W}_n$
is non-degenerate. We consider the following commutative diagram to obtain
our results:

\begin{center}
$\begin{CD}
\widetilde{W}_{n-k}\times S_k @>Ind>> \widetilde{W}_n\\
@VVIndV @VVIndV\\
W_{n-k}\times S_k @>Ind>> W_n
\end{CD}$
\end{center}

\begin{enumerate}
\item Suppose $(\alpha^0,\alpha^1)\vdash (n-k)$ with
$\alpha^0\neq\alpha^1$ and
$(\beta^0,\beta^1)\vdash n$ with $\beta^0\neq\beta^1$. Additionally let
$\delta\vdash k$. Then the multiplicity of $\chi^{\{\beta^0,\beta^1\}}$ in
$Ind_{\widetilde{W}_{n-k}\times
S_k}^{\widetilde{W}_n}(\chi^{\{\alpha^0,\alpha^1\}}\otimes\chi^\delta)$ is equal
to the multiplicity of $\chi^{(\beta^0,\beta^1)}$ plus the multiplicity of
$\chi^{(\beta^1,\beta^0)}$ in $Ind_{W_{n-k}\times
S_k}^{W_n}(\chi^{(\alpha^0,\alpha^1)}\otimes\chi^\delta)$.
\newline
\item Suppose that $(\alpha,\beta)\vdash (n-k)$ with
$\alpha\neq\beta$ and $(\alpha,\alpha)\vdash n$.
Additionally let $\delta\vdash k$. Then the multiplicity of
$\chi^{\{\alpha,\alpha\}}$ in $Ind_{\widetilde{W}_{n-k}\times
S_k}^{\widetilde{W}_n}(\chi^{\{\alpha,\beta\}}\otimes\chi^\delta)$ is equal to
the multiplicity of $\chi^{(\alpha,\alpha)}$ in $Ind_{W_{n-k}\times
S_k}^{W_n}(\chi^{(\alpha,\beta)}\otimes\chi^\delta)$. The multiplicity of
$\chi'^{\{\alpha,\alpha\}}$ is exactly the same.
\newline
\item Suppose that $(\alpha,\alpha)\vdash (n-k)$ and
$(\alpha,\beta)\vdash n$ with $\alpha\neq\beta$.
Additionally let $\delta\vdash k$. Then the multiplicity of
$\chi^{\{\alpha,\beta\}}$ in $Ind_{\widetilde{W}_{n-k}\times
S_k}^{\widetilde{W}_n}(\chi^{\{\alpha,\alpha\}}\otimes\chi^\delta)$ is equal to
the multiplicity of $\chi^{(\alpha,\beta)}$ in $Ind_{W_{n-k}\times
S_k}^{W_n}(\chi^{(\alpha,\alpha)}\otimes\chi^\delta)$. The same statement is
true with $\chi^{\{\alpha,\alpha\}}$ replaced with $\chi'^{\{\alpha,\alpha\}}$.
\end{enumerate}
\end{description}

\section{The Brauer Homomorphism}

Let $G$ be a finite group and $M$ an $\mathcal{O}G$-module. The Brauer
homomorphism $Br_P^G(M)$ of $M$ with respect to some $p$-subgroup $P$
of $G$ is the natural surjection:

\begin{center}
$Br_P^G:M^P\rightarrow M^P/(\sum_{Q<P}Tr_Q^P(M^Q)+\mathcal{J}M^P)$
\end{center}

If $M$ is indecomposable with trivial source then $Br_P^G(M^P)\neq0$ if and only
if $M$ has a vertex containing $P$.
\newline
\newline
We consider the specific case where $M=\mathcal{O}G$ with the action of $G$
given by conjugation. In this case $Br_P^G(M^P)\cong kC_G(P)$ as
$kN_G(P)$-modules with $Br_P^G$ given by:

\begin{center}
$Br_P^G(\sum_{g\in G}\alpha_gg)=\sum_{g\in C_G(P)}\overline{\alpha_g}g$
\end{center}

Now if $H$ is a subgroup of $G$ containing $N_G(P)$ and $b$ and $c$ are block
idempotents of $\mathcal{O}G$ and $\mathcal{O}H$ respectively such that the
corresponding blocks both have defect group $P$, then $\mathcal{O}Hc$ is
described as the Brauer correspondent of $\mathcal{O}Gb$ in $H$ if
$Br_P^G(b)=Br_P^H(c)$. When $H=N_G(P)$ we describe $\mathcal{O}Hc$ simply as the
Brauer correspondent.

\section{The Finite Classical Groups}

Let $q$ be a power of a prime. We describe each of the finite classical
groups $G_m(q)$ in question in two ways. First we view $G_m(q)$ as the fixed
points of some algebraic group $\mathbf{G_m(q)}$ under a Frobenius morphism $F$
and secondly as the group of linear maps that preserve or scale some bilinear
form over $\mathbb{F}_q$. For both these descriptions we follow G. Hiss and K.
Kessar\cite[Section 4]{hiskes2000}.

\subsection{Forms}

Let $W$ be an $m$-dimensional vector space over $\mathbb{F}_{q}$. We describe
the following bilinear forms on $W$. We require $q$ to be odd for the symplectic and orthogonal forms only.

\begin{description}
\item[Unitary] To define a unitary form we require $q$ to be a square, say
$q=q_0^2$. We want a non-degenerate sesquilinear form on $W$. Up to isomorphism
there is one such form:

\begin{center}
$<u,v>=u_1v_m^{q_0}+\dots+u_mv_1^{q_0}$
\end{center}

\item[Symplectic] We want a non-degenerate anti-symmetric bilinear form on $W$.
In this case $m$ must be even. Say $m=2n$. Up to isomorphism there is one such
form:

\begin{center}
$<u,v>=(u_1v_{2n}+\dots+u_nv_{n+1})-(u_{n+1}v_n+\dots+u_{2n}v_1)$
\end{center}

\item[Orthogonal] We want a non-degenerate symmetric bilinear form on $W$. There
are two non-isomorphic such forms:

\begin{center}
$<u,v>=u_1v_m+u_2v_{m-1}+\dots+u_{m-1}v_2+u_mv_1$
\newline
\newline
and
\newline
\newline
$<u,v>=u_1v_1+\delta u_mv_m+(u_2v_{m-1}+u_3v_{m-2}+\dots+u_{m-2}v_3+u_{m-1}v_2)$
\newline
\newline
where $-\delta$ is a non-square in $\mathbb{F}_q$.
\end{center}

These two forms are known as type 1 and type -1 respectively.
\end{description}

Throughout this section if $W$ is a vector space over $K$ with some bilinear
form $g$ then we introduce the sets\cite[Section 1]{fonsri1989}:
\begin{align*}
I(W,g)&=\{x\in GL(W)|g(xu,xv)=g(u,v),\forall u,v\in W\}
\\
I_0(W,g)&=I(W,g)\cap SL(W)
\end{align*}

When $\dim(W)$ is even, say $2n$, then we define:

\begin{align*}
J(W,g)&=\{x\in GL(W)|\exists\lambda_x\in K,g(xu,xv)=\lambda_xg(u,v),\forall
u,v\in W\}
\\
J_0(W,g)&=\{x\in J(W,g)|\det(x)=\lambda_x^n\}
\end{align*}

\subsection{Description of Groups}

First we fix an algebraic closure $\overline{\mathbb{F}_q}$ of $\mathbb{F}_q$
and introduce the automorphism:

\begin{align*}
GL_n(\overline{\mathbb{F}_q})&\rightarrow GL_n(\overline{\mathbb{F}_q})\\
M&\mapsto M^{[q]}
\end{align*}

Where $M^{[q]}$ is obtained from the matrix $M$ by raising all its entries to
the power $q$. In all cases except where stated $F$ will be the above
automorphism.
\newline
\newline
When $q$ is odd we fix a non-square $\delta\in\mathbb{F}_q$. We will use $J_n$ to denote the
$n\times n$ matrix with $1$s along the antidiagonal and $0$s elsewhere,
$J'_{2n}$ the matrix
$\left(\begin{array}{cc}
  0  & J_n \\
-J_n & 0   \\
\end{array}\right)$ and $J''_{2n}$ the matrix
$\left(\begin{array}{cccc}
 0 & 0 & 0 & J_{n-1} \\
 0 & 1 & 0 & 0 \\
 0 & 0 & \delta & 0 \\
 J_{n-1} & 0 & 0 & 0 \\
\end{array}\right)$.
\newline
\newline
Let $\mathbf{V}$ be an $m$-dimensional vector space over
$\overline{\mathbb{F}_q}$ with basis $\{e_1,\dots,e_m\}$ and $V$ the
$\mathbb{F}_q$-vector space given by the $\mathbb{F}_q$-span of
$\{e_1,\dots,e_m\}$. $f$ will be a bilinear form on $V$ but we can also consider
it as a bilinear form $\mathbf{f}$ on $\mathbf{V}$ in a natural way.
\newline
\newline
We allow $q$ to be even in our description of $GL_n(q)$ and $U_n(q)$ but assume $q$ is odd for all the other groups.
This convention will hold throughout this paper.
\newline
\begin{description}
\item[$\mathbf{GL_n(q)}$] Let $\mathbf{G}=GL_n(\overline{\mathbb{F}_q})$. Then
$G=GL_n(q)=\mathbf{G}^F$.
\newline
$GL_n(q)$ has order $q^{\frac{n(n-1)}{2}}\prod_{i=1}^n(q^i-1)$.
\newline
\item[$\mathbf{U_n(q)}$] Here we require $q$ to be a square, say $q=q_0^2$. We
set $\mathbf{G}=GL_n(\overline{\mathbb{F}_q})$ and set $F$ to be the
automorphism:

\begin{align*}
F: M\mapsto J_nM^{-t[q_0]}J_n
\end{align*}

Then $G=U_n(q)=\mathbf{G}^F$.
\newline
\newline
Alternatively if $m=n$ and

\begin{align*}
f(u,v)=u_1v_n^{q_0}+\dots+u_nv_1^{q_0}
\end{align*}

then $U_n(q)=I(V,f)=I(\mathbf{V},\mathbf{f})$.
\newline
\newline
$U_n(q)$ has order $q_0^{\frac{n(n-1)}{2}}\prod_{i=1}^n(q_0^i-(-1)^i)$\cite[Section 2.6]{wall1962}.
\newline
\item[$\mathbf{Sp_{2n}(q)}$, $\mathbf{CSp_{2n}(q)}$]
\begin{align*}
Sp_{2n}(\overline{\mathbb{F}_q})&=\{x\in
GL_{2n}(\overline{\mathbb{F}_q})|x^tJ'_{2n}x=J'_{2n}\}
\\
CSp_{2n}(\overline{\mathbb{F}_q})&=\{x\in
GL_{2n}(\overline{\mathbb{F}_q})|x^tJ'_{2n}x=\lambda_xJ'_{2n},
\lambda_x\in\overline{\mathbb{F}_q}\}
\end{align*}

For $\mathbf{G}=Sp_{2n}(\overline{\mathbb{F}_q})$,
$CSp_{2n}(\overline{\mathbb{F}_q})$ we have $G=\mathbf{G}^F=Sp_{2n}(q)$,
$CSp_{2n}(q)$ respectively.
\newline
\newline
Alternatively if $m=2n$ and

\begin{align*}
f(u,v)=(u_1v_{2n}+\dots+u_nv_{n+1})-(u_{n+1}v_n+\dots+u_{2n}v_1)
\end{align*}

then

\begin{align*}
Sp_{2n}(q)&=I(V,f) & Sp_{2n}(\overline{\mathbb{F}_q})&=I(\mathbf{V},\mathbf{f})
\\
CSp_{2n}(q)&=J(V,f) &
CSp_{2n}(\overline{\mathbb{F}_q})&=J(\mathbf{V},\mathbf{f})
\end{align*}
$CSp_{2n}(q)$ has order $q^{n^2}(q-1)\prod_{i=1}^n(q^{2i}-1)$\cite[Section
2.6]{wall1962}.
\newline
\item[$\mathbf{O_{2n+1}(q)}$, $\mathbf{SO_{2n+1}(q)}$, $\mathbf{CO_{2n+1}(q)}$,
$\mathbf{CSO_{2n+1}(q)}$]
\begin{align*}
O_{2n+1}(\overline{\mathbb{F}_q})&=\{x\in
GL_{2n+1}(\overline{\mathbb{F}_q})|x^tJ_{2n+1}x=J_{2n+1}\}
\\
SO_{2n+1}(\overline{\mathbb{F}_q})&=O_{2n+1}(\overline{\mathbb{F}_q})\cap
SL_{2n+1}(\overline{\mathbb{F}_q})
\end{align*}

For $\mathbf{G}=O_{2n+1}(\overline{\mathbb{F}_q})$,
$SO_{2n+1}(\overline{\mathbb{F}_q})$ we have $G=\mathbf{G}^F=O_{2n+1}(q)$,
$SO_{2n+1}(q)$ respectively.
\newline
\newline
Alternatively if $m=2n+1$ and

\begin{align*}
f(u,v)=u_1v_{2n+1}+\dots+u_{2n+1}v_1
\end{align*}

then

\begin{align*}
O_{2n+1}(q)&=I(V,f) &
O_{2n+1}(\overline{\mathbb{F}_q})&=I(\mathbf{V},\mathbf{f})
\\
SO_{2n+1}(q)&=I_0(V,f) &
SO_{2n+1}(\overline{\mathbb{F}_q})&=I_0(\mathbf{V},\mathbf{f})
\end{align*}

$SO_{2n+1}(q)$ has order $q^{n^2}\prod_{i=1}^n(q^{2i}-1)$\cite[Section
2.6]{wall1962}.
\newline
\newline
We note that when $\dim(V)$ is odd although there are two non-isomorphic non-degenerate orthogonal
forms on $V$ they produce isomorphic groups.
\newline
\item[$\mathbf{O^+_{2n}(q)}$, $\mathbf{SO^+_{2n}(q)}$, $\mathbf{CO^+_{2n}(q)}$, $\mathbf{CSO^+_{2n}(q)}$]
\begin{align*}
O^+_{2n}(\overline{\mathbb{F}_q})&=\{x\in
GL_{2n}(\overline{\mathbb{F}_q})|x^tJ_{2n}x=J_{2n}\}
\\
SO^+_{2n}(\overline{\mathbb{F}_q})&=O^+_{2n}(\overline{\mathbb{F}_q})\cap
SL_{2n}(\overline{\mathbb{F}_q})
\end{align*}

For $\mathbf{G}=O^+_{2n}(\overline{\mathbb{F}_q})$,
$SO^+_{2n}(\overline{\mathbb{F}_q})$ we have $G=\mathbf{G}^F=O^+_{2n}(q)$,
$SO^+_{2n}(q)$ respectively.

\begin{align*}
CO^+_{2n}(\overline{\mathbb{F}_q})&=\{x\in
GL_{2n}(\overline{\mathbb{F}_q})|x^tJ_{2n}x=\lambda_xJ_{2n},
\lambda_x\in\overline{\mathbb{F}_q}\}
\\
CSO^+_{2n}(\overline{\mathbb{F}_q})&=\{x\in
CO^+_{2n}(\overline{\mathbb{F}_q})|\lambda_x^n=det(x)\}
\end{align*}

For $\mathbf{G}=CO^+_{2n}(\overline{\mathbb{F}_q})$,
$CSO^+_{2n}(\overline{\mathbb{F}_q})$ we have $G=\mathbf{G}^F=CO^+_{2n}(q)$,
$CSO^+_{2n}(q)$ respectively.
\newline
\newline
Alternatively if $m=2n$ and

\begin{align*}
f(u,v)=u_1v_{2n}+\dots+u_{2n}v_1
\end{align*}

then

\begin{align*}
O^+_{2n}(q)&=I(V,f) &
O^+_{2n}(\overline{\mathbb{F}_q})&=I(\mathbf{V},\mathbf{f})
\\
SO^+_{2n}(q)&=I_0(V,f) &
SO^+_{2n}(\overline{\mathbb{F}_q})&=I_0(\mathbf{V},\mathbf{f})
\\
CO^+_{2n}(q)&=J(V,f) &
CO^+_{2n}(\overline{\mathbb{F}_q})&=J(\mathbf{V},\mathbf{f})
\\
CSO^+_{2n}(q)&=J_0(V,f) &
CSO^+_{2n}(\overline{\mathbb{F}_q})&=J_0(\mathbf{V},\mathbf{f})
\end{align*}

$SO_{2n}^+(q)$ has order
$q^{n(n-1)}(q^n-1)\prod_{i=1}^{n-1}(q^{2i}-1)$\cite[Section 2.6]{wall1962}.
\newline
$CSO_{2n}^+(q)$ has order
$q^{n(n-1)}(q^n-1)(q-1)\prod_{i=1}^{n-1}(q^{2i}-1)$\cite[Section 2.6]{wall1962}.
\newline
\item[$\mathbf{O^-_{2n}(q)}$, $\mathbf{SO^-_{2n}(q)}$, $\mathbf{CO^-_{2n}(q)}$,
$\mathbf{CSO^-_{2n}(q)}$]
\begin{align*}
O^-_{2n}(\overline{\mathbb{F}_q})&=\{x\in
GL_{2n}(\overline{\mathbb{F}_q})|x^tJ''_{2n}x=J''_{2n}\}
\\
SO^-_{2n}(\overline{\mathbb{F}_q})&=O^-_{2n}(\overline{\mathbb{F}_q})\cap
SL_{2n}(\overline{\mathbb{F}_q})
\end{align*}

For $\mathbf{G}=O^-_{2n}(\overline{\mathbb{F}_q})$,
$SO^-_{2n}(\overline{\mathbb{F}_q})$ we have $G=\mathbf{G}^F=O^-_{2n}(q)$,
$SO^-_{2n}(q)$ respectively.

\begin{align*}
CO^-_{2n}(\overline{\mathbb{F}_q})&=\{x\in
GL_{2n}(\overline{\mathbb{F}_q})|x^tJ''_{2n}x=\lambda_xJ''_{2n},
\lambda_x\in\overline{\mathbb{F}_q}\}
\\
CSO^-_{2n}(\overline{\mathbb{F}_q})&=\{x\in
CO^-_{2n}(\overline{\mathbb{F}_q})|\lambda_x^n=det(x)\}
\end{align*}

For $\mathbf{G}=CO^-_{2n}(\overline{\mathbb{F}_q})$,
$CSO^-_{2n}(\overline{\mathbb{F}_q})$ we have $G=\mathbf{G}^F=CO^-_{2n}(q)$,
$CSO^-_{2n}(q)$ respectively.
\newline
\newline
Alternatively if $m=2n$ and

\begin{align*}
f(u,v)=u_1v_{2n}+\dots+u_nv_n+\delta u_{n+1}v_{n+1}+\dots+u_{2n}v_1
\end{align*}

then

\begin{align*}
O^-_{2n}(q)&=I(V,f) &
O^-_{2n}(\overline{\mathbb{F}_q})&=I(\mathbf{V},\mathbf{f})
\\
SO^-_{2n}(q)&=I_0(V,f) &
SO^-_{2n}(\overline{\mathbb{F}_q})&=I_0(\mathbf{V},\mathbf{f})
\\
CO^-_{2n}(q)&=J(V,f) &
CO^-_{2n}(\overline{\mathbb{F}_q})&=J(\mathbf{V},\mathbf{f})
\\
CSO^-_{2n}(q)&=J_0(V,f) &
CSO^-_{2n}(\overline{\mathbb{F}_q})&=J_0(\mathbf{V},\mathbf{f})
\end{align*}

$SO_{2n}^-(q)$ has order
$q^{n(n-1)}(q^n+1)\prod_{i=1}^{n-1}(q^{2i}-1)$\cite[Section 2.6]{wall1962}.
\newline
$CSO_{2n}^-(q)$ has order
$q^{n(n-1)}(q^n+1)(q-1)\prod_{i=1}^{n-1}(q^{2i}-1)$\cite[Section 2.6]{wall1962}.
\end{description}

\subsection{Clifford Groups}\label{sec:cli}

For our description of Clifford groups we follow P. Fong and B.
Srinivasan\cite[Section 2]{fonsri1989}. Let $V$ be a vector space over $K$ of
finite dimension greater than $1$ endowed with a non-degenerate quadratic form
$Q$. Let $<,>$ be the corresponding symmetric bilinear form. The Clifford
algebra $\mathcal{C}(V)$ is the $K$-algebra generated by $V$ subject to all the
linear relations in $V$ as well as the additional condition:

\begin{center}
$v^2=Q(v)1$ for all $v\in V$
\end{center}

$\mathcal{C}(V)$ has a $\mathbb{Z}_2$-grading given by demanding all the
non-zero $v\in V$ are odd. We use $\mathcal{C}_+(V)$ to denote the even part.
Now define the Clifford group $Cl(V)$ and special Clifford group $Cl_+(V)$ as
follows:

\begin{align*}
Cl(V)=&\{x\in\mathcal{C}(V)^\times|xVx^{-1}=V\}
\\
Cl_+(V)=&\{x\in\mathcal{C}_+(V)^\times|xVx^{-1}=V\}
\end{align*}

Now if $x\in Cl(V)$ and $v\in V$ then:

\begin{align*}
Q(xvx^{-1})=(xvx^{-1})^2=xv^2x^{-1}=xQ(v)x^{-1}=Q(v)
\end{align*}

Therefore conjugation by an element in $Cl(V)$ preserves $<,>$ on $V$. This
gives us the map:

\begin{align*}
Cl(V)\rightarrow O(V)
\end{align*}

where $O(V)$ is the group of linear maps on $V$ that preserve $<,>$.
\newline
\newline
When we restrict to $Cl_+(V)$ we get the short exact sequence:

\begin{center}
$\begin{CD}
0 @> >> K^\times @> >> Cl_+(V) @>\pi>> SO(V)\longrightarrow0
\end{CD}$
\end{center}

Let $q$ be the power of an odd prime.
We use $Cl_{2n+1}(q)$ (respectively $Cl^+_{2n}(q)$, $Cl^-_{2n}(q)$) to denote
the special Clifford groups defined over $\mathbb{F}_q$ with respect to the
bilinear form given by the Gram matrix $J_{2n+1}$ (respectively $J_{2n}$,
$J''_{2n}$).
\newline
\newline
Similarly we can define $Cl_{2n+1}(\overline{\mathbb{F}_q})$,
$Cl^+_{2n}(\overline{\mathbb{F}_q})$ and $Cl^-_{2n}(\overline{\mathbb{F}_q})$
over $\overline{\mathbb{F}_q}$.
\newline
\newline
If $\mathbf{V}$ is the $\overline{\mathbb{F}_q}$-vector space with basis
$\{e_1\dots e_m\}$ then we have an $\mathbb{F}_q$-vector space automorphism of
$\mathbf{V}$ given by:

\begin{align*}
\sum \alpha_i e_i\mapsto\sum \alpha^q_i e_i
\end{align*}

This extends to an automorphism $F$ of $Cl_{2n+1}(\overline{\mathbb{F}_q})$
(respectively $Cl^+_{2n}(\overline{\mathbb{F}_q})$,
$Cl^+_{2n}(\overline{\mathbb{F}_q})$) (where $\{e_1\dots e_m\}$ is the basis
with respect to which the Gram matrices are taken). Then we have:

\begin{align*}
Cl_{2n+1}(q)&=Cl_{2n+1}(\overline{\mathbb{F}_q})^F
\\
Cl^+_{2n}(q)&=Cl^+_{2n}(\overline{\mathbb{F}_q})^F
\\
Cl^-_{2n}(q)&=Cl^-_{2n}(\overline{\mathbb{F}_q})^F
\end{align*}

\subsection{Levi Subgroups}\label{sec:levi}
We will now describe certain subgroups of the finite classical and Clifford
groups described above as in\cite[Section 4.4]{hiskes2000}.

\begin{description}
\item[$\mathbf{GL_n(q)}$] If we have positive integers $(n_1,n_2,\dots,n_t)$
with $\sum_in_i=n$ then
$GL_{n_1}(q)\times\dots\times GL_{n_t}(q)$ embeds in $GL_n(q)$ in the natural
way:
\\
\begin{align*}
GL_{n_1}(q)\times\dots\times GL_{n_t}(q)&\hookrightarrow GL_n(q)\\
A_1\times\dots\times A_t&\mapsto \left(\begin{array}{ccc}
 A_1 & & \\
 & \ddots & \\
 & & A_t
\end{array}\right)
\end{align*}
\newline
\newline
We will use $L_{(n_1,\dots,n_t)}(q)$ to denote this Levi subgroup.
\newline
\item[$\mathbf{U_n(q)}$, $\mathbf{Sp_{2n}(q)}$, $\mathbf{CSp_{2n}(q)}$,
$\mathbf{SO_{2n+1}(q)}$, $\mathbf{SO^\pm_{2n}(q)}$, $\mathbf{CSO^\pm_{2n}(q)}$]
Let $G_m(q)$ be one of the above groups and $(n_1,n_2,\dots,n_t)$ positive
integers with $2s<m$ where $s=\sum_in_i$, then $GL_{n_1}(q)\times\dots\times
GL_{n_t}(q)\times G_{m-2s}(q)$ embeds in $G_m(q)$ via the map:
\\
\begin{align*}
GL_{n_1}(q)\times\dots\times GL_{n_t}(q)\times G_{m-2s}(q)&\hookrightarrow
G_m(q)\\
A_1\times\dots\times A_t\times B&\mapsto \left(\begin{array}{ccccccc}
 A_1 & & & & & & \\
 & \ddots & & & & & \\
 & & A_t & & & & \\
 & & & B & & & \\
 & & & & A_t' & & \\
 & & & & & \ddots & \\
 & & & & & & A_1' \\
\end{array}\right)
\end{align*}

where the $A_i'$ are chosen to stay in $G_m(q)$. So
$A_i'=J_{n_i}A_i^{-t[q_0]}J_{n_i}$ for the unitary group and
$A_i'=\lambda_BJ_{n_i}A_i^{-t[q]}J_{n_i}$ in all other cases. Note that
$\lambda_B$ is only needed for the conformal groups and we set $\lambda_B=1$ for
all the other groups. We will use $L_{m,(n_1,\dots,n_t)}(q)$ to denote this Levi subgroup of $G_m(q)$.
\newline
\item[$\mathbf{Cl_{2n+1}(q)}$, $\mathbf{Cl^\pm_{2n}(q)}$] Let $Cl^{(\pm)}_m(q)$
denote one of the above groups and $SO^{(\pm)}_m(q)$ the corresponding special
orthogonal group giving us the exact sequence.

\begin{align*}
0\longrightarrow\mathbb{F}_q\longrightarrow Cl^{(\pm)}_m(q)&\longrightarrow
SO^{(\pm)}_m(q)\longrightarrow0\\
\end{align*}

Then if we have positive integers $(n_1,n_2,\dots,n_t)$ with $2s<m$ where
$\sum_in_i=s$ then we have $L_{m,(n_1,\dots,n_t)}(q)\cong
GL_{n_1}(q)\times\dots\times GL_{n_t}(q)\times SO^{(\pm)}_{m-2s}(q)$ is a Levi
subgroup of $SO^{(\pm)}_m(q)$ as described above. The pre-image of
$L_{m,(n_1,\dots,n_t)}(q)$ in $Cl^{(\pm)}_m(q)$ is isomorphic to
$GL_{n_1}(q)\times\dots\times GL_{n_t}(q)\times Cl^{(\pm)}_{m-2s}(q)$ and we
also denote it by $L_{m,(n_1,\dots,n_t)}(q)$.
\end{description}

\section{Representation Theory of the Finite Classical Groups}

\subsection{Semisimple elements}\label{sec:semi}

Semisimple elements will be important when we describe the ordinary characters
of the finite classical groups. By a semisimple element we mean a matrix whose
minimal polynomial has no repeated roots. Equivalently they are the matrices
that are diagonalisable over some field extension. If $\Gamma$ is a polynomial
let $d_\Gamma$ denote the degree of $\Gamma$. Finally we let $\Delta$ be the set
of monic irreducible polynomials over $\mathbb{F}_q$ with non-zero roots. We
follow P. Fong and B. Srinivasen in\cite[Section 1]{fonsri1982} and\cite[Section
1]{fonsri1989} for our description of semisimple elements.
\newline
\begin{description}
\item[$\mathbf{GL_n(q)}$] A conjugacy class of semisimple elements in $GL_n(q)$
is completely defined by
their common characteristic polynomial. If $\Gamma\in\Delta$ let $m_\Gamma(s)$
be the multiplicity of $\Gamma$ in the characteristic polynomial of $s$. So we
have the function

\begin{align*}
\Delta &\rightarrow \mathbb{N}_0\\
\Gamma &\mapsto m_\Gamma(s)
\end{align*}

with $\sum_{\Gamma\in\Delta} m_\Gamma(s).d_\Gamma = n$. Conversely any such
function uniquely defines a semisimple conjugacy class in $GL_n(q)$.
\newline
\item[$\mathbf{U_n(q)}$] A conjugacy class of semisimple elements in $U_n(q)$ is
also completely defined
by their common characteristic polynomial. However, if $s\in U_n(q)$ and
$\omega$ is a root of the characteristic polynomial of $s$ then $\omega^{-q_0}$
is also a root occurring with the same multiplicity. With that in mind we define
the following involution on $\Delta$:

\begin{align*}
\Gamma=&(X^m+a_{m-1}X^{m-1}+\dots+a_1X+a_0)\mapsto\\
\overline{\Gamma}=&{a_0}^{-q_0}({a_0}^{q_0}X^m+{a_1}^{q_0}X^{m-1}+\dots+{a_{m-1}
}^{q_0}X+1)
\end{align*}

So $\overline{\Gamma}$ is the unique monic polynomial with roots those of
$\Gamma$ raised to the power $-q_0$.

\begin{align*}
\Lambda_1=&\{\Gamma|\Gamma\in\Delta, \overline{\Gamma}=\Gamma\}\\
\Lambda_2=&\{\Gamma\overline{\Gamma}|\Gamma\in\Delta,
\overline{\Gamma}\neq\Gamma\}\\
\Lambda=&\Lambda_1\cup\Lambda_2
\end{align*}

If $\Gamma\in\Lambda$ let $m_\Gamma(s)$ be the multiplicity of $\Gamma$ in the
characteristic polynomial of $s$. So we have the function

\begin{align*}
\Lambda &\rightarrow \mathbb{N}_0\\
\Gamma &\mapsto m_\Gamma(s)
\end{align*}

with $\sum_{\Gamma\in\Lambda} m_\Gamma(s).d_\Gamma=n$. Conversely any such
function uniquely defines a semisimple conjugacy class in $U_n(q)$.
\newline
\item[$\mathbf{Sp_{2n}(q)}$] A conjugacy class of semisimple elements in
$Sp_{2n}(q)$ is also completely defined by their common characteristic
polynomial. However, if $s\in Sp_{2n}(q)$ and $\omega$ is a root of the
characteristic polynomial of $s$ then $\omega^{-1}$ is also a root occurring
with the same multiplicity. With that in mind we define the following involution
on $\Delta$:

\begin{align*}
\Gamma=&(X^m+a_{m-1}X^{m-1}+\dots+a_1X+a_0)\mapsto\\
\widetilde{\Gamma}=&{a_0}^{-1}(a_0X^m+a_1X^{m-1}+\dots+a_{m-1}X+1)
\end{align*}

So $\widetilde{\Gamma}$ is the unique monic polynomial with roots the inverses
of those of $\Gamma$. Next we define the sets:

\begin{align*}
\Phi_0=&\{X-1,X+1\}\\
\Phi_1=&\{\Gamma|\Gamma\in\Delta\backslash\Phi_0, \widetilde{\Gamma}=\Gamma\}\\
\Phi_2=&\{\Gamma\widetilde{\Gamma}|\Gamma\in\Delta,
\widetilde{\Gamma}\neq\Gamma\}\\
\Phi=&\Phi_0\cup\Phi_1\cup\Phi_2
\end{align*}

If $\Gamma\in\Phi$ let $m_\Gamma(s)$ be the multiplicity of $\Gamma$ in the
characteristic polynomial of $s$. So we have the function

\begin{align*}
\Phi &\rightarrow \mathbb{N}_0\\
\Gamma &\mapsto m_\Gamma(s)
\end{align*}

with $m_{X+1}(s)$ even and $\sum_{\Gamma\in\Phi} m_\Gamma(s).d_\Gamma = 2n$.
Conversely any such function uniquely defines a semisimple conjugacy class in
$Sp_{2n}(q)$.
\newline
\item[$\mathbf{SO^{(\pm)}_m(q)}$] Let $SO^{(\pm)}_m(q)$ be any one of our
special orthogonal groups. Once again if
$s\in SO^{(\pm)}_m(q)$ and $\omega$ is a root of the characteristic polynomial
of $s$ then $\omega^{-1}$ is also a root occurring with the same multiplicity.
So again we have the function:

\begin{align*}
\Phi &\rightarrow \mathbb{N}_0\\
\Gamma &\mapsto m_\Gamma(s)
\end{align*}

with $\sum_{\Gamma\in\Phi} m_\Gamma(s).d_\Gamma = m$ and $m_{X+1}(s)$ even. This
does not uniquely determine a conjugacy class however. If both $m_{X-1}(s)$ and
$m_{X+1}(s)$ are non-zero this function defines $2$ conjugacy classes, otherwise
we just get $1$ conjugacy class. Later on the semisimple elements we consider
will all have $m_{X+1}(s)=0$.
\newline
\item[$\mathbf{Cl^{(\pm)}_m(q)}$]\cite[Section 2]{fonsri1989} We have the
surjection described in~\ref{sec:cli}:

\begin{center}
$\pi:Cl^{(\pm)}_m(q)\twoheadrightarrow SO^{(\pm)}_m(q)$
\end{center}

An element of $Cl^{(\pm)}_m(q)$ is described as semisimple if its image in
$SO^{(\pm)}_m(q)$ is semisimple. Let $\pi(t)=s\in SO^{(\pm)}_m(q)$ be semisimple
and $C$ its conjugacy class in $SO^{(\pm)}_m(q)$. If $m_{X-1}(s)$ and
$m_{X+1}(s)$ are both non-zero then $\pi^{-1}(C)$ is the union of
$\frac{q-1}{2}$ conjugacy classes in $Cl^{(\pm)}_m(q)$ and $t$ is conjugate to
$-t$. Otherwise $\pi^{-1}(C)$ is the union of $q-1$ conjugacy classes in
$Cl^{(\pm)}_m(q)$ and no $2$ distinct pre-images of $s$ are conjugate.
\newline
From now on we will use $m_\Gamma(t)$ to denote $m_\Gamma(s)$.
\end{description}

For any semisimple element $s$ we have a corresponding decomposition of
$V=\mathbb{F}_q^m$. For each $\Gamma\in\Phi$ (or $\Delta$ for $GL_n(q)$ or
$\Lambda$ for $U_n(q)$) we define $V_\Gamma$ to be the null space of $\Gamma(s)$
(or $\Gamma(\pi(s))$ in the case of the special Clifford groups).

\begin{align*}
V=\oplus_\Gamma V_\Gamma
\end{align*}

is then an orthogonal decomposition of $V$.

\subsection{Characters of Finite Classical groups}\label{sec:char}

We will describe labels for the ordinary irreducible characters of some of the
finite classical groups. We follow P. Fong and B. Srinivasan in\cite[Section
1]{fonsri1982} and\cite[Section 2]{fonsri1989} for the description of the
characters. We leave out $Sp_{2n}(q)$ and $SO^\pm(q)$, the appropriate characters
for these groups will be described in section~\ref{sec:spso}. First of all we
will describe a subset of all the ordinary characters called the unipotent
characters.
\newline
\newline
The unipotent characters of $GL_n(q)$ and $U_n(q)$ are both labeled by
partitions of $n$ (for example $(n)$ corresponds to the trivial character of
both groups).
\newline
\newline
The unipotent characters of $CSp_{2n}(q)$ and $SO_{2n+1}(q)$ are both labeled by
symbols with rank $n$ and odd defect.
\newline
\newline
The unipotent characters of $CSO^+_{2n}(q)$ are labeled by symbols with rank $n$
and defect $\equiv0$ (mod $4$) with the added rule that degenerate symbols label
$2$ characters.
\newline
\newline
The unipotent characters of $CSO^-_{2n}(q)$ are labeled by symbols with rank $n$
and defect $\equiv2$ (mod $4$).
\newline
\newline
Consider the following table:

\begin{center}
\begin{tabular}{| c | c |} \hline
$G$ & $G^*$ \\ \hline
$GL_n(q)$ & $GL_n(q)$ \\
$U_n(q)$ & $U_n(q)$ \\
$SO_{2n+1}(q)$ & $Sp_{2n}(q)$ \\
$CSp_{2n}(q)$ & $Cl_{2n+1}(q)$ \\
$CSO^+_{2n}(q)$ & $Cl^+_{2n}(q)$ \\
$CSO^-_{2n}(q)$ & $Cl^-_{2n}(q)$ \\ \hline
\end{tabular}
\end{center}

Let $s\in G^*$ be semisimple with corresponding decomposition $V=\oplus_\Gamma
V_\Gamma$. We define $\Psi_\Gamma(s)$ as follows:
\newline
If $G^*$ is a general linear or unitary group then $\Psi_\Gamma(s)=\{$partitions
of $m_\Gamma(s)\}$.
\newline
Now we assume that $G^*$ is not a general linear or unitary group.
\newline
For $\Gamma\in\Phi_0$ let $\Psi_\Gamma(s)=\{$symbols of rank $\lfloor
\frac{m_\Gamma(s)}{2}\rfloor\}$ subject to the following conditions:

\begin{enumerate}
\item If the form induced on $V_\Gamma$ is symplectic or orthogonal of odd
dimension, then the symbols have odd defect.
\item If the form induced on $V_\Gamma$ is orthogonal of even dimension and type
1, then the symbols have defect $\equiv0$ (mod $4$). Moreover, degenerate
symbols are counted twice. If $\lambda$ is such a degenerate symbol then we say
$\lambda,\lambda'\in\Psi_\Gamma(s)$.
\item If the form induced on $V_\Gamma$ is orthogonal of even dimension and type
-1, then the symbols have $\equiv2$ (mod $4$).
\end{enumerate}

For $\Gamma\in\Phi_1\cup\Phi_2$ let $\Psi_\Gamma(s)=\{$partitions of
$m_\Gamma(s)\}$.
\newline
Now we set $\Psi(s)=\prod_\Gamma\Psi_\Gamma(s)$.
\newline
\newline
The characters of $G$ are labeled by a semisimple element $s$ of $G^*$ together
with a $\lambda\in\Psi(s)$. We denote this character $\chi_{s,\lambda}$.
$\chi_{s,\lambda}$ and $\chi_{t,\mu}$ represent the same character if and only
if $s$ is conjugate to $t$ in $G^*$ and $\lambda=\mu$.
\newline
\newline
Now if $s$ does not have $-1$ as an eigenvalue then we can label
$\chi_{s,\lambda}$ by a unipotent character of $C_{G^*}(s)^*$\cite[Section
1]{fonsri1982}, \cite[Section 4]{fonsri1989}.
\newline
We will now describe $C_{G^*}(s)$ for all relevant $G$ to make this labeling clear.

\begin{description}
\item[$\mathbf{GL_n(q)}$] Let $G^*=GL_n(q)$ and let $s\in G^*$ be semisimple.
Then:
\begin{gather*}
C_{G^*}(s)\cong\prod_{\Gamma\in\Delta}GL_{m_\Gamma(s)}(q^{d_\Gamma})
\end{gather*}

\item[$\mathbf{U_n(q)}$] Let $G^*=U_n(q)$ and let $s\in G^*$ be semisimple.
Then:

\begin{gather*}
C_{G^*}(s)\cong (\prod_{\Gamma\in
\Phi_1}U_{m_\Gamma(s)}(q^{d_\Gamma}))\times(\prod_{\Gamma\in
\Phi_2}GL_{m_\Gamma(s)}(q^\frac{d_\Gamma}{2}))
\end{gather*}

\item[$\mathbf{Sp_{2n}(q)}$, $\mathbf{Cl_{2n+1}(q)}$, $\mathbf{Cl^\pm_{2n}(q)}$]
Let $G^*=G_m(q)$ be one of the above groups. Let $s\in G^*$ be semisimple.

\begin{gather*}
C_{G^*}(s)\cong G_{m_{X-1}(s)}(q)\times(\prod_{\Gamma\in
\Phi_1}U_{m_\Gamma(s)}(q^{d_\Gamma}))\times(\prod_{\Gamma\in
\Phi_2}GL_{m_\Gamma(s)}(q^{d_\Gamma}))
\end{gather*}
\end{description}

\section{Harish-Chandra Induction}\label{sec:har}

\subsection{Harish-Chandra Induction}

Let $p$ be an odd prime not dividing $q$ and $(K,\mathcal{O},k)$ a $p$-modular
system as introduced in the introduction~\ref{intro}. Let $G=G_m(q)$ be
$GL_n(q)$, $U_n(q)$, $Sp_{2n}(q)$, $CSp_{2n}(q)$, $SO_{2n+1}(q)$,
$SO^+_{2n}(q)$, $SO^-_{2n}(q)$, $CSO^+_{2n}(q)$ or $CSO^-_{2n}(q)$ and
$L=L_{m,(n_1,\dots,n_t)}(q)$ (or $L=L_{(n_1,\dots,n_t)}(q)$ in the case of the
general linear group) as described in section~\ref{sec:levi}. We describe a
functor
$HCInd_L^G$ from $\mathcal{O}L$-mod to $\mathcal{O}G$-mod\cite[Example
4.6(iii)]{digmic1991}.
\newline
\newline
We set $U$ to be the subgroup of $G$ consisting of matrices of the form:

\begin{gather*}
\left(\begin{array}{ccc}
 I_{n_1} & \dots & *\\
 & \ddots & \vdots\\
 & & I_{n_t}\\
\end{array}\right)
\end{gather*}

for the general linear group and

\begin{gather*}
\left(\begin{array}{ccccccc}
 I_{n_1} & \dots & * & * & * & \dots & * \\
 & \ddots & \vdots & \vdots & \vdots & \ddots & \vdots \\
 & & I_{n_t} & * & * & \dots & * \\
 & & & I_{m-2s} & * & \dots & * \\
 & & & & I_{n_t} & \dots & * \\
 & & & & & \ddots & \vdots \\
 & & & & & & I_{n_1} \\
\end{array}\right)
\end{gather*}

otherwise.
\newline
\newline
Now in all cases $U$ has $p'$ order. In other words $|U|$ is invertible in
$\mathcal{O}$. Let $U^+$ be the idempotent $\frac{1}{|U|}\sum_{u\in U}u$ in
$\mathcal{O}G$. $L$ normalises $U$ and so commutes with $U^+$ and so
$\mathcal{O}GU^+$ is an $(\mathcal{O}G$,$\mathcal{O}L)$-bimodule. Now set
$HCInd_L^G$ to be the functor $\mathcal{O}GU^+\otimes_{\mathcal{O}L}-$.
\newline
\newline
If $b$ is a central idempotent of $\mathcal{O}G$ and $c$ a central idempotent of
$\mathcal{O}L$. We
will want the functor $\mathcal{O}GbU^+\otimes_{\mathcal{O}Lc}-$ from
$\mathcal{O}Lc$-mod to $\mathcal{O}Gb$-mod. We denote this functor
$HCInd_{L,c}^{G,b}$.

\subsection{Characters under Harish-Chandra Induction}

Harish-Chandra induction is a functor from $\mathcal{O}L$-mod to
$\mathcal{O}G$-mod. However, we want to know what happens at the level of
characters viewing Harish-Chandra induction as a functor from $KL$-mod to
$KG$-mod. Let $R_L^G$ denote Harish-Chandra induction on characters. Again if
$b$ is a central idempotent of $\mathcal{O}G$ and $c$ a central idempotent of
$\mathcal{O}L$ then we denote by:

\begin{gather*}
R_{L,c}^{G,b}:Irr(\mathcal{O}Lc)\rightarrow Irr(\mathcal{O}Gb)
\end{gather*}

the corresponding function on characters.
\newline
\newline
In this section we only consider $L=L_{m,(k)}(q)$ (or $L=L_{(k,n-k)}(q)$ for the
general linear group). The reason for this is that Harish-Chandra induction is
transitive. This means that if $L$ has more general linear factors then the
functor $R_L^G$ can be calculated iteratively.

\subsubsection{Unipotent Characters}

First we will first look at the effect of Harish-Chandra induction on unipotent
characters for the groups $GL_n(q)$, $U_n(q)$, $CSp_{2n}(q)$, $SO_{2n+1}(q)$,
$CSO^\pm_{2n}(q)$\cite[Section 5.3]{hiskes2000}.
\newline
\newline
\begin{description}
\item[$\mathbf{GL_n(q)}$] If $G=GL_n(q)$ and $L=L_{(k,n-k)}\cong GL_k(q)\times
GL_{n-k}(q)$ then the
multiplicity of $\chi_{1,\gamma}$ in
$R_L^G(\chi_{1,\alpha}\otimes\chi_{1,\beta})$ is $g_{\alpha,\beta}^\gamma$
(see~\ref{sec:LitRic}).
\newline
\item[$\mathbf{U_n(q)}$, $\mathbf{CSp_{2n}(q)}$, $\mathbf{SO_{2n+1}(q)}$, $\mathbf{CSO^-_{2n}(q)}$] Let $G=G_m(q)$ be one of the above groups
and $L=L_{m,(k)}(q)\cong GL_k(q)\times G_{m-2k}(q)$.
We use the alternative description of partitions and symbols described in section~\ref{not:char}.
\newline
\newline
In all four groups the multiplicity of $\chi_{1,(s,\mu,\nu)}$ in
$R_L^G(\chi_{1,\gamma}\otimes\chi_{1,(t,\alpha,\beta)})$ is $0$ unless $s=t$ and
in this case it is equal to the multiplicity of $\chi^{(\mu,\nu)}$ in
$Ind_{S_k\times W_{v-k}}^{W_v}(\chi^{\gamma}\otimes\chi^{(\alpha,\beta)})$ for
an appropriate $v$ (see~\ref{sec:weyl}).
\newline
\item[$\mathbf{CSO^+_{2n}(q)}$] For $CSO^+_{2n}(q)$ we again use the description
in~\ref{not:char}. We then do
exactly the same calculation as above unless $s=0$. In this case the calculation
is then carried out in a Weyl group of type $D_m$ instead of $B_m$. In other
words the multiplicity of $\chi_{1,(0,\mu,\nu)}$ in
$R_L^G(\chi_{1,\gamma}\otimes\chi_{1,(0,\alpha,\beta)})$ is equal to the
multiplicity of $\chi^{\{\mu,\nu\}}$ in $Ind_{S_k\times
\widetilde{W}_{v-k}}^{\widetilde{W}_v}(\chi^\gamma\otimes\chi^{\{\alpha,\beta\}}
)$ for an appropriate $v$. Note that $\chi_{1,(0,\mu,\mu)}$,
$\chi'_{1,(0,\mu,\mu)}$ correspond to $\chi^{\{\mu,\mu\}}$,
$\chi'^{\{\mu,\mu\}}$ respectively.
\end{description}

\subsubsection{General Case}

Now we drop the assumption that the characters are unipotent. However, we keep
the assumption that our semisimple labels don't have $-1$ as an eigenvalue. Let
$s\in L^*$ be semisimple. We can then use~\ref{sec:char} to calculate
$R_L^G(\chi_{s,\lambda})$ by passing to Harish-Chandra induction for unipotent
characters from $C_{L^*}(s)^*$ to $C_{G^*}(s)^*$\cite[Section 5.4]{hiskes2000}.

\begin{description}
\item[$\mathbf{GL_n(q)}$] Let $G=GL_n(q)$ and $L=L_{(k,n-k)}\cong GL_k(q)\times
GL_{n-k}(q)$. If
$\chi_{s_1,\lambda_1}\otimes\chi_{s_2,\lambda_2}$ is a character of $L$ then

\begin{gather*}
C_{L^*}(s_1\times
s_2)\cong\prod_{\Gamma\in\Delta}(GL_{m_\Gamma(s_1)}(q^{d_\Gamma})\times
GL_{m_\Gamma(s_2)}(q^{d_\Gamma}))
\end{gather*}

So for each $\Gamma\in\Delta$ we do Harish-Chandra induction to get a sum of
unipotent character of

\begin{gather*}
GL_{m_\Gamma(s_1)+m_\Gamma(s_2)}(q^{d_\Gamma})
\end{gather*}

We now have a sum of unipotent characters for $C_{G^*}(s_1\times s_2)^*$ and
hence a character of $G$. This character is
$R_L^G(\chi_{s_1,\lambda_1}\otimes\chi_{s_2,\lambda_2})$.
\newline
\item[$\mathbf{U_n(q)}$] Let $G=U_n(q)$ and $L=L_{n,(k)}\cong GL_k(q)\times
U_{n-2k}(q)$. If
$\chi_{s_1,\lambda_1}\otimes\chi_{s_2,\lambda_2}$ is a character of $L$ then

\begin{center}
$C_{L^*}(s_1\times
s_2)\cong\prod_{\Gamma\in\Lambda_1}(GL_{m_\Gamma(s_1)}(q^{d_\Gamma})\times
U_{m_\Gamma(s_2)}(q^{d_\Gamma}))\times$\\
$\prod_{\Gamma\overline{\Gamma}\in\Lambda_2}(GL_{m_\Gamma(s_1)}(q^{d_\Gamma}
)\times
GL_{m_{\overline{\Gamma}}(s_1)}(q^{d_\Gamma})
\times GL_{m_{\Gamma\overline{\Gamma}}(s_2)}(q^{d_\Gamma}))$
\end{center}

Then for each $\Gamma\in\Lambda_1$ we do Harish-Chandra induction to get a sum
of unipotent characters of

\begin{center}
$U_{2m_\Gamma(s_1)+m_\Gamma(s_2)}(q^{d_\Gamma})$
\end{center}

and for each $\Gamma\in\Lambda_2$ we do Harish-Chandra induction twice to get a
sum of unipotent characters of

\begin{center}
$GL_{m_\Gamma(s_1)+m_{\overline{\Gamma}}(s_1)+m_{\Gamma\overline{\Gamma}}(s_2)}
(q^{d_\Gamma})$
\end{center}

We now have a sum of unipotent characters of $C_{G^*}(s_1\times s_2)^*$ and
hence a character of $G$. This character is
$R_L^G(\chi_{s_1,\lambda_1}\otimes\chi_{s_2,\lambda_2})$.
\newline
\item[$\mathbf{SO_{2n+1}(q)}$, $\mathbf{CSp_{2n}(q)}$, $\mathbf{CSO^+_{2n}(q)}$,
$\mathbf{CSO^-_{2n}(q)}$] Let $G=G_m(q)$ be $SO_{2n+1}(q)$, $CSp_{2n}(q)$,
$CSO^+_{2n}(q)$ or
$CSO^-_{2n}(q)$ and $L=L_{m,(k)}\cong GL_k(q)\times G_{m-2k}(q)$. If
$\chi_{s_1,\lambda_1}\otimes\chi_{s_2,\lambda_2}$ is a character of $L$ then

\begin{center}
$C_{L^*}(s_1\times s_2)\cong(GL_{m_{X-1}(s_1)}(q)\times
G^*_{m_{X-1}(s_2)}(q))\times$\\
$\prod_{\Gamma\in\Psi_1}(GL_{m_\Gamma(s_1)}(q^{d_\Gamma})\times
U_{{m_\Gamma}(s_1)}(q^{d_\Gamma}))$\\
$\prod_{\Gamma\widetilde{\Gamma}\in\Psi_2}(GL_{m_\Gamma(s_1)}(q^{d_\Gamma}
)\times
GL_{m_{\widetilde{\Gamma}}(s_1)}(q^{d_\Gamma})
\times GL_{m_{\Gamma\widetilde{\Gamma}}(s_2)}(q^{d_\Gamma}))$
\end{center}

Once again we do Harish-Chandra induction on unipotent characters to obtain a
sum of unipotent characters of $C_{G^*}(s_1\times s_2)^*$ and hence a character
of $G$. This character is
$R_L^G(\chi_{s_1,\lambda_1}\otimes\chi_{s_2,\lambda_2})$.
\end{description}

\subsection{Characters with Equal Dimension}\label{sec:dim}

We can use Harish-Chandra induction to show that some pairs of characters have
the same dimension. This will be useful later when we will be performing some
calculations.

\begin{lem}
The following pairs of characters have the same dimensions.
\begin{enumerate}
\item $\chi_{s,\lambda}$ and $\chi_{s^{-q_0},\lambda^{-q_0}}$ of $GL_k(q)$ where
$q=q_0^2$ and $\lambda^{-q_0}_{\Gamma}=\lambda_{\overline{\Gamma}}$
(see~\ref{sec:semi}).
\item $\chi_{s,\lambda}$ and $\chi_{s^{-1},\lambda^{-1}}$ of $GL_k(q)$ where
$\lambda^{-1}_{\Gamma}=\lambda_{\widetilde{\Gamma}}$ (see section~\ref{sec:semi}).
\item $\chi_{s,\lambda}$ and $\chi_{s,\lambda'}$ of $CSO^+_{2n}(q)$ where
$s\in\mathbb{F}_q$ and $\lambda$ is a degenerate symbol.
\end{enumerate}
\end{lem}

\begin{proof}
We prove all three results by considering Harish-Chandra induction of a pair of
characters from $L$ to $G$ that give the same character and noting that the
dimension is always multiplied by $[G:LU]$ (see~\ref{sec:har}).
\begin{enumerate}
\item We set $G=U_{n+2k}(q)$ and $L=GL_k(q)\times U_n(q)$. We Harish-Chandra
induce $\chi_{s,\lambda}\otimes\chi$ and
$\chi_{s^{-q_0},\lambda^{-q_0}}\otimes\chi$ where $\chi$ is any unipotent
character of $U_(q)$.
\item We set $G=CSp_{2(n+k)}(q)$ and $L=GL_k(q)\times CSp_{2n}(q)$. We
Harish-Chandra induce $\chi_{s,\lambda}\otimes\chi$ and
$\chi_{s^{-1},\lambda^{-1}}\otimes\chi$ where $\chi$ is any unipotent character
of $Sp_{2n}(q)$.
\item We set $G=CSO^+_{2(n+1)}(q)$ and $L=GL_1(q)\times CSO^+_{2n}(q)$. We
Harish-Chandra induce $\chi_{1,(1)}\otimes\chi_{s,\lambda}$ and
$\chi_{1,(1)}\otimes\chi_{s,\lambda'}$.
\end{enumerate}
\end{proof}

\section{Unipotent Blocks of Finite Classical groups}\label{sec:blocks}

We continue with the assumption that $p$ is an odd prime not dividing $q$. We
will describe a subset of the $p$-blocks of our finite classical groups called
the unipotent blocks. First we need the notion of linear and unitary
prime\cite[Section 6.1]{hiskes2000}. Let $G=GL_n(q)$, $U_n(q)$, $Sp_{2n}(q)$,
$CSp_{2n}(q)$, $SO_{2n+1}(q)$, $SO^\pm_{2n}(q)$ or $CSO^\pm_{2n}(q)$.
\newline
\newline
Let $d$ be the multiplicative order of $q\mod{p}$ and $p^a$ the maximum power
of $p$ dividing $q^d-1$. If $G=GL_n(q)$ let $e=d$. If $G=U_n(q)$ let $e$ be the
multiplicative order of $-q_0\mod{p}$. In this case we have:

\begin{itemize}
\item $d=e$ if $e$ is odd
\item $d=\frac{1}{2}e$ if $e$ is even
\end{itemize}

If $G$ is any of the other seven groups define $e$ to be the multiplicative
order of $q^2$ (mod $p$). In these cases we have:

\begin{itemize}
\item $e=d$ if $d$ is odd
\item $e=\frac{1}{2}d$ if $d$ is even
\end{itemize}

When $G$ is not a general linear group we have the notion of $p$ being a linear
or unitary prime. When we work with the unitary group $p$ is unitary if $e=d$
and linear otherwise. For all the other groups $p$ is linear if $e=d$ and
unitary otherwise. From now on we will assume $p$ is linear with respect to $q$.

\subsection{Blocks}

Let us now restrict our attention to $G=GL_n(q)$, $U_n(q)$, $CSp_{2n}(q)$,
$SO_{2n+1}(q)$ or $CSO^\pm_{2n}(q)$. We follow\cite[Section
6.2]{hiskes2000}. We will describe the unipotent blocks of $Sp_{2n}(q)$ and
$SO^\pm_{2n}(q)$ at the end of this section.
\newline
\newline
First we introduce the notion of a $\phi_d$-torus where $\phi_d$ is the $d$th
cyclotomic polynomial. A $\phi_d$-torus of an algebraic group $\mathbf{G}$ is an
$F$ stable torus $\mathbf{T}$ whose polynomial order is $(\phi_d)^t$ for some
$t$. Now a $d$-split Levi subgroup $\mathbf{K}$ of $\mathbf{G}$ is a subgroup of
the form $C_\mathbf{G}(\mathbf{T})$ for such a $\mathbf{T}$. A character of
$\mathbf{K}^F$ is described as $d$-cuspidal if it does not appear as a summand
of a character Harish-Chandra induced up from any $\mathbf{K'}^F$ where
$\mathbf{K'}<\mathbf{K}$ is a $d$-split Levi subgroup of $\mathbf{K}$.
\newline
\newline
The unipotent blocks of $\mathbf{G}^F$ are labeled by conjugacy classes (in $G$)
of pairs $(\mathbf{K},\psi)$ where $\mathbf{K}$ is a $d$-split Levi subgroup of
$\mathbf{G}$ and $\psi$ is a $d$-cuspidal character of $\mathbf{K}^F$.
\newline
\newline
Now we take $G=G_m(q)=\mathbf{G}^F$ to be one of our six groups and $p$ a linear
prime with respect to $q$ where appropriate. A typical unipotent block of $G$ is
labeled by $(\mathbf{K},\psi)$ where $\mathbf{K}^F\cong GL_1(q^d)^t\times
G_{m'}(q)$ and $\psi=(1_{GL_1(q^d)})^t\otimes\chi$ where $m'=m-dt$ when $G$ is
the general linear group and $m-2dt$ otherwise and $\chi$ is a unipotent
character of $G_{m'}(q)$ whose label is an $e$-core.
\newline
\newline
An alternative way to describe the $p$-blocks, where $p$ is a linear prime with
respect to $q$, is as follows. This description can be found in more detail
in\cite[Theorem D]{fonsri1982} and \cite[Theorem 10B, Theorem 11E]{fonsri1989}.
First take a unipotent character of $G$. Associated to this unipotent character
we have either a partition or symbol of which we take the $e$-core. Our
unipotent block is then just labeled by this $e$-core $\mu$.
\newline
\newline
We can pair our two descriptions up by letting $\mu$ in the second description
be the label of $\chi$ in the first. We call this block $B_\mu$.
\newline
\newline
Note that for $G=CSO^+_{2n}(q)$, $\chi$ labeled by a degenerate symbol and $t>0$
then $\mu$ would be both copies of this degenerate symbol. However, the two
corresponding characters of $G=CSO^+_{2(n-dt)}(q)$ are conjugate in $G$ and
hence both labels label the same block of $G$.

\subsection{Defect Groups and Dual Defect Groups}\label{sec:defect}

Consider the block of $G$ labeled by $(\mathbf{K},\psi)$. Any Sylow $p$-subgroup
of $C_\mathbf{G}^0([\mathbf{K},\mathbf{K}])^F$ is a defect group for the
block\cite[Section 6.2]{hiskes2000}.
\newline
P. Fong and B. Srinivasan describe the concept of dual defect groups. If $D\leq
G$ is a defect group of $(\mathbf{K},\psi)$ then $D^*$ is naturally a subgroup
of $G^*$. For $GL_n(q)$ and $U_n(q)$ we identify $G$ with $G^*$ and $D$ with
$D^*$. For the other four groups see\cite[Section 12, Section 13]{fonsri1989}.
$D^*$ is then described as a dual defect group for $(\mathbf{K},\psi)$.

\subsection{Characters in Unipotent Blocks}\label{sec:charblock}

We now describe the characters in a unipotent block\cite[Section 7]{fonsri1982},
\cite[Section 12, Section 13]{fonsri1989}. If $\Gamma\in\Phi$ (or $\Delta$ for
$GL_n(q)$ or $\Lambda$ for $U_n(q)$) then we set $e_\Gamma$ to be the additive
of $d_\Gamma$ (mod $d$). Let $G_m(q)$ be one of our groups and let $B_\mu$ be a
unipotent block of $G_m(q)$, $\chi_{t,\lambda}$ an irreducible character of
$G_m(q)$ and fix a dual defect group $D^*$ of $B_\mu$. Then $\chi_{t,\lambda}$
lies in $B_\mu$ if and only if:

\begin{enumerate}
\item $t$ is conjugate to $x$ for some $x\in D^*$
\item The $e$-core of $\lambda_{X-1}$ is a subset of $\mu$
\item The $e_\Gamma$-core of $\lambda_\Gamma$ is empty for all other $\Gamma$
\end{enumerate}

\subsection{Unipotent Blocks of \texorpdfstring{$Sp_{2n}(q)$}{TEXT} and \texorpdfstring{$SO_{2n}^\pm(q)$}{TEXT}}\label{sec:spso}

Let $G=Sp_{2n}(q)$ (respectively $SO_{2n}^\pm(q)$) and $\hat{G}=CSp_{2n}(q)$
(respectively $CSO_{2n}^\pm(q)$). We now deal with the unipotent blocks of
$\mathcal{O}G$ via the following lemma:

\begin{lem}\label{lem:so2n}
The unipotent blocks of $\mathcal{O}G$ are in one-to-one correspondence with
those of
$\mathcal{O}\hat{G}$. Furthermore if $i$ and $j$ are corresponding block
idempotents then we have the
following isomorphism of $\mathcal{O}$-algebras:

\begin{align*}
\mathcal{O}Z(\hat{G})_p\otimes_\mathcal{O}\mathcal{O}Gi\cong\mathcal{O}\hat{G}j
\end{align*}

The isomorphism is given by multiplication by $j$ and the correspondence of
characters from $\mathcal{O}\hat{G}j$ to
$\mathcal{O}Z(\hat{G})_p\otimes_\mathcal{O}\mathcal{O}Gi$ is given by
restriction.
\newline
Also if $D$ is a defect group of $\mathcal{O}Gi$ then $Z(\hat{G})_p\times D$
is a defect group of $\mathcal{O}\hat{G}j$.
\end{lem}

Before we prove the above we mention that we will label the block corresponding
to $i$ with the same label (symbol) as the block corresponding to $j$.

\begin{proof}
Apply\cite[Theorem 12]{cabeng1993} to
$\mathbf{G}=Sp_{2n}(\overline{\mathbb{F}_q})$ and
$CSp_{2n}(\overline{\mathbb{F}_q})$ (respectively
$SO^\pm_{2n}(\overline{\mathbb{F}_q})$ and
$CSO^\pm_{2n}(\overline{\mathbb{F}_q})$). The lemma then follows from chasing
the appropriate character correspondences and then applying\cite[Proposition
6]{cabeng1993}.
\end{proof}

\section{Main Theorem}\label{sec:thm}

Let $G=G_m(q)$ be one of the following groups:

\begin{enumerate}
\item $U_n(q)$
\item \begin{enumerate} 
      \item $Sp_{2n}(q)$
      \item $CSp_{2n}(q)$
      \end{enumerate}
\item $SO_{2n+1}(q)$
\item \begin{enumerate}
      \item $SO^\pm_{2n}(q)$
      \item $CSO^\pm_{2n}(q)$
      \end{enumerate}
\end{enumerate}

We continue with our assumption that $q$ is odd in all except case (1) where it can be even or odd.
As in previous sections let $p$ an odd prime not dividing $q$. We also assume that $p$ is a
linear prime with respect to $q$. $d$ and $e$ will have the meaning given in~\ref{sec:blocks}.
\newline
\newline
For each $w$ $(0\leq w<p)$ we wish to find an integer $m=m(w)$ and a unipotent
block $B_\rho$ of $G_m(q)$ such that theorem~\ref{thm:intro} holds. We now
describe conditions for $m$ and $\rho$ to satisfy in each of the four cases.

\begin{enumerate}\label{cond:rho}
\item[1] $\rho$ is an $e$-core partition with a representation on a
$2d$-abacus such that the $i$th runner has at least $w-1$ fewer beads than the
$(i+2)$th runner for $(0\leq i\leq 2d-3)$. If $r$ is the rank of $\rho$ then
$m=r+2dw$.
\item[2,3,4] There exists a $2d$-linear diagram of $\rho$ such that
the $i$th runner has at least
$w-1$ fewer beads than the $(i+1)$th runner for $(0\leq i\leq d-2)$ and $(d\leq
i\leq 2d-2)$. $\rho$ has non-zero rank $r$ and $m=2(r+dw)+1$ in case 3 and
$2(r+dw)$ in cases 2 and 4.
\end{enumerate}

We now fix $m$ and $\rho$ along with a $2d$-abacus representation of $\rho$ such
that the above
property holds and every position on the first $w$ rows is occupied with a bead.
\newline
\newline
We need to set up some notation that is required for us to state and prove the main theorem.
\newline
\newline
We set $\widetilde{G}=\widetilde{G}_m(q)$ to be $O_{2n+1}(q)$ (respectively
$O^\pm_{2n}(q)$, $CO^\pm_{2n}(q)$) when $G=SO_{2n+1}(q)$ (respectively
$SO^\pm_{2n}(q)$, $CSO^\pm_{2n}(q)$) and $G$ otherwise. Let $T_i$ be the subgroup
of $\widetilde{G}$ generated by

\begin{gather*}
\left(\begin{array}{ccccc}
 I_{d(i-1)} &               &          &              &\\
                &          &          & I_d &\\
                &          & I_{m-2di} &               &\\
                & -I_d &          &          &\\
                &          &        &      & I_{d(i-1)}\\
\end{array}\right)
\end{gather*}

in case 2 and by

\begin{gather*}
\left(\begin{array}{ccccc}
 I_{d(i-1)} &               &          &              &\\
                &          &          & I_d &\\
                &          & I_{m-2di} &               &\\
                & I_d &          &          &\\
                &          &        &      & I_{d(i-1)}\\
\end{array}\right)
\end{gather*}

in cases 1,3 and 4.
\newline
\newline
Let $G=G_0>G_1>\dots>G_w=L$ be a sequence of Levi subgroups of $G$ where
$G_i=L_{m,(d^i)}(q)\cong GL_d(q)_1\times\dots\times GL_d(q)_i\times
G_{m-2di}(q)$ as described in~\ref{sec:levi}.
\newline
\newline
We denote by $H$ the subgroup
$(GL_d(q)_1.T_1)\times\dots\times(GL_d(q)_w.T_w)\times\widetilde{G}_{m-2dw}(q)$ of
$\widetilde{G}_m(q)$. Note that if we adopt our notation from the introduction
we have $H$ is naturally ismorphic to $(GL_d(q).2)^w\times\widetilde{G}_{m-2dw}(q)$.
\newline
\newline
Now consider the Levi subgroup $L_{m,(dw)}(q)\cong GL_{dw}(q)\times
G_{m-2dw}(q)\geq L$ of $G$. Let $S$ be the subgroup of permutation matrices of
$GL_{dw}(q)\leq L_{m,(dw)}(q)$ whose conjugation action permutes the $GL_d(q)_i$s. Clearly $S$
normalises $H$ and intersects it trivially so $H.S\cong((GL_d(q).2)\wr
S_w)\times\widetilde{G}_{m-2dw}(q)$. We set $N=H.S\cap G$ and have in all cases that $|N|=2^ww!|L|$.
\newline
\newline
For $(0<i\leq w)$ let $a_i$ be the principal block idempotent of $\mathcal{O}GL_d(q)_i$ and for $(0\leq i\leq w)$ let
$f_{w-i}$ be the unipotent block idempotent of $\mathcal{O}G_{m-2di}(q)$ associated with
the partition $\rho$. We set $b_i$ to be the block idempotent $a_1\otimes\dots\otimes a_i\otimes f_{w-i}$
of $\mathcal{O}G_i$. We set $b=b_0$ and $f=b_w$. When $\rho$ is degenerate in case 4 we
denote by $f_0$ and $f'_0$ the $2$ blocks of $G_{m-2dw}(q)$ labeled by $\rho$ and by $f=b_w$ and $f'=b'_w$ the block
idempotents $a_1\otimes\dots\otimes a_w\otimes f_0$ and $a_1\otimes\dots\otimes
a_w\otimes f'_0$ of $\mathcal{O}L$ respectively.
\newline
\newline
Fix a Sylow $p$-subgroup $R$ of $GL_d(q)$ (note that $|R|=p^a$
see~\ref{sec:blocks} for the meaning of $a$) and let $P_1\times\dots\times P_w$
be $w$ copies of $R$, one in each $GL_d(q)_i$ of $GL_d(q)_1\times\dots\times
GL_d(q)_w$. We also set $Z$ to be the subgroup of $G$ consisting of scalar
matrices and $P=Z_p.(P_1\times\dots\times P_w)$. $P$ is a defect group for
$\mathcal{O}Gb$ (see~\ref{sec:defect}) and $N_G(P)\leq N$.
\newline
\newline
Additionally for $(1\leq i\leq w)$ let $U_i$ be the subgroup of matrices of
$G_{i-1}$ of the form:

\begin{gather*}
\left(\begin{array}{ccccc}
 I_{d(i-1)} &               &          &               &\\
                & I_d   & *   & * &\\
                &          & I_{m-di} & *              &\\
                &          &          & I_d &\\
                &          &        &      & I_{d(i-1)}\\
\end{array}\right)
\end{gather*}

And set $U_i^{+}= \frac{1}{|U_{i}|} \sum_{u\in U_i}u$.

\subsection{Non-Degenerate Case}

In this subsection we assume that $\rho$ is non-degenerate. We restrict our
attention to cases 1, 2(b), 3 and 4(b). We will later prove the corresponding
theorem for cases 2(a) and 4(a) using~\ref{lem:so2n}.

\begin{thm}\label{thm:nondegen}
$\mathcal{O}Nf$ is a block of $\mathcal{O}Nf$ and is Morita
equivalent to $\mathcal{O}Gb$.
\end{thm}

For the the proof of this theorem, which will fill this section,
we follow W. Turner\cite[Section 2]{turner2002}. We will need a
number of lemmas first.

\begin{lem}\label{lem:omn}$ $
\begin{enumerate}
\item $P$ is defect group for $\mathcal{O}G_ib_i$ for $(0\leq i\leq w)$.
\item $Br_P^G(b_i)=a'_1\otimes\dots\otimes a'_w\otimes f_0$ where $a'_i$ is the
principal block idempotent of $C_{GL_d(q)_i}(P_i)$ $(0\leq i\leq w)$. Also
$Br_P^G(U_i^+)=1$.
\item $N$ stabilizes $f$ and as an $\mathcal{O}(N\times L)$-module,
$\mathcal{O}Nf$ is indecomposable with vertex $\Delta(P)$. In particular,
$\mathcal{O}Nf$ is a block of $N$.
\item $\mathcal{O}Gb$ and $\mathcal{O}Nf$ both have defect group $P$ and are
Brauer correspondents.
\end{enumerate}
\end{lem}

\begin{proof}$ $
\begin{enumerate}
\item $P_i$ is defect group for $\mathcal{O}GL_d(q)_ia_i$ and
$(P_{i+1}\times\dots\times P_w).Z_p$ is a defect group for
$\mathcal{O}G_{m-2di}(q)f_{w-i}$ (see~\ref{sec:defect}).
\item $C_G(P)<G_i$ so $Br_P^G(b_i)=Br_P^{G_i}(b_i)$
    \newline
    $Br_P^{G_i}(b_i)=Br_{P_1}^{GL_d(q)_1}(a_1)\otimes\dots\otimes
Br_{P^i}^{GL_d(q)_i}(a_i)\otimes Br_{(P_{i+1}\times\dots\times
P_w).Z_p}^{GL_d(q)_{i+1}\times\dots\times GL_d(q)_w\times
G_{m-2dw}(q)}(f_{w-i})$
    \newline
    and from\cite[Theorem 3.2]{bromic1989} we see that we must get a block of
the form: $a'_1\otimes\dots\otimes a'_w\otimes\varepsilon$
    \newline
    where $\varepsilon$ is a sum of unipotent block idempotents of $G_{m-2dw}(q)$
of defect zero.
    \newline
    Secondly\cite[Lemma 4.5]{cabeng1994} tells us that $a'_1\otimes\dots\otimes
a'_w\otimes f_0$ must appear as a constituent.
    \newline
    Finally we see that no other block idempotent of $C_G(P)$ can appear as
otherwise we would have two distinct block idempotents $\alpha$ and $\beta$ of
$G_i$ with defect group $P$ and $Br_P^{G_i}(\alpha)Br_P^{G_i}(\beta)\neq0$. This
is of course a contradiction as $Br_P^{G_i}$ is an algebra homomorphism.
    \newline
    The second part is clear.
\item $N$ clearly stabilizes $f$ in all cases except case 4(b).
The only thing to check in this case is that
$CO^\pm_{m-2dw}(q)$ stabilizes $f_0$. This is clear however, by looking at
part (2) and noting that conjugation by $N_G(P)$ commutes with $Br_P^G$.
    \newline
    By part(1), $\mathcal{O}Lf$ has vertex $\Delta(P)$. Since $C_G(P)\leq L$,
the conjugate of $\Delta(P)$ by an element of $N\times L$ outside $L\times L$ is
never conjugate to $\Delta(P)$ in $L\times L$. Consequently, the stabilizer of
$\mathcal{O}Lf$ in $N\times L$ is exactly $L\times L$. So if
$\mathcal{O}Nf=Ind_{L\times L}^{N\times L}(\mathcal{O}Lf)$ were decomposable
then $Ind_{L\times L}^{N\times L}(\mathcal{O}Lf)$ would have an indecomposable
summand, as a $\mathcal{O}(N\times L)$-module, whose restriction to $L\times L$
has every summand without a vertex contained in $\Delta(P)$. Thus this
indecomposable $\mathcal{O}(N\times L)$-module does not have a vertex contained
in $\Delta(P)$. This is of course a contradiction and so $\mathcal{O}Nf$ is
indecomposable as a $\mathcal{O}(N\times L)$-module. Its vertex is clearly
contained in $\Delta (P)$ and its restriction to $L\times L$ has a summand with
vertex $\Delta(P)$. So $\mathcal{O}Nf$ has vertex $\Delta(P)$ as a
$\mathcal{O}(N\times L)$-module.
\item $P$ is a defect group for $\mathcal{O}Gb$ by part (1). Secondly we note
that any $p$-subgroup of $N$ is contained in $L$ ($L\lhd N$ and $p\nmid[N:L]$).
This tells us that $\mathcal{O}Lf$ has the same defect group as $\mathcal{O}Nf$.
So we have that $\mathcal{O}Nf$ has defect group $P$. Finally we have that
$Br_P^N(f)=Br_P^G(b)$ by (2) and so $\mathcal{O}Nf$ is the Brauer correspondent
of $\mathcal{O}Gb$ in $N$.
\end{enumerate}
\end{proof}

By Alperin's description of the Brauer correspondence\cite[Chapter 14, Theorem
2]{alperi1986} the $\mathcal{O}(G\times G)$-module $\mathcal{O}Gb$ and the
$\mathcal{O}(N\times N)$-module $\mathcal{O}Nf$ both have vertex $\Delta(P)$ and
are Green correspondents. Let $X$ be the Green correspondent of $\mathcal{O}Gb$
in $G\times N$. Then $X$ is the unique indecomposable summand of $Res_{G\times
N}^{G\times G}(\mathcal{O}Gb)$ with vertex $\Delta (P)$ and $\mathcal{O}Nf$ is
the unique indecomposable summand of $Res_{N\times N}^{G\times N}(X)$ with
vertex $\Delta (P)$. It is then clear that $bX=X$ and that $Xf\neq0$, and so
$Xf=X$ and $X$ is an $(\mathcal{O}Gb$,$\mathcal{O}Nf)$-bimodule.
\newline
\newline
Let $Y$ be
$_GY_L=\mathcal{O}Gb_0U_1^+b_1\dots U_w^+b_w$, an
$(\mathcal{O}Gb$,$\mathcal{O}Lf)$-bimodule. So the functor
$Y\otimes_{\mathcal{O}L}-$ from $\mathcal{O}L$-mod to $\mathcal{O}G$-mod is
$HCInd_{G_1b_1}^{G_0b_0}\dots HCInd_{G_1b_w}^{G_0b_{w-1}}$.

\begin{prop}\label{prop:spilt}
There is a sequence of $\mathcal{O}$-split monomorphisms of algebras

\begin{align*}
\mathcal{O}Nf\hookrightarrow \operatorname{End}_{\mathcal{O}G}(X)\hookrightarrow \operatorname{End}_{\mathcal{O}G}(Y)
\end{align*}

Also the left $\mathcal{O}Gb$-module $X$ is a progenerator for $\mathcal{O}Gb$.
\end{prop}

\begin{proof}
$_G\mathcal{O}Gb_G$ is isomorphic to a direct summand of $Ind_{G\times
N}^{G\times G}(_GX_N)$ as they are Green correspondents. Thus $_G\mathcal{O}Gb$
is a direct summand of $[G:N]$ copies of $_GX$ and $_GX$ is a progenerator for
$\mathcal{O}Gb$.
\newline
\newline
Now there is an $\mathcal{O}$-split homomorphism of algebras
$\mathcal{O}Nf\rightarrow \operatorname{End}_{\mathcal{O}G}(X)$ given by multiplying on the
right of $X$. Since $_N\mathcal{O}Nf_N$ is a direct summand of $Res_{N\times
N}^{G\times N}(_GX_N)$ this homomorphism is an $\mathcal{O}$-split monomorphism.
\newline
\newline
Next $Res_{G\times L}^{G\times N}(_GX_N)$ is indecomposable with vertex $\Delta
(P)$. First, $G\times L$ contains $\Delta (P)$ so $_GX_N$ is a direct summand of
$Ind_{G\times L}^{G\times N}(Res_{G\times L}^{G\times N}(_GX_N))$. Since
$(G\times L)\unlhd (G\times N)$ there exists an indecomposable summand $M$ of
$Res_{G\times L}^{G\times N}(_GX_N)$ such that $Res_{G\times L}^{G\times
N}(_GX_N)$ is a direct sum of conjugates of $M$ in $G\times N$. It is possible
to pick a set of coset representatives of $(G\times L)$ in $(G\times N)$ that
all normalise $\Delta (P)$. So $Res_{G\times L}^{G\times N}(_GX_N)$ is the sum
of indecomposable modules all with vertex $\Delta (P)$.
\newline
Secondly, $Res_{G\times L}^{G\times
N}(_GX_N)=_GX_N\otimes_{\mathcal{O}N}\mathcal{O}Nf_L$ is a direct summand of
$Ind_{N\times L}^{G\times L}(\mathcal{O}Nf)$ which by Green correspondence has
exactly one summand with vertex $\Delta (P)$. So $_GX_L$ is indecomposable with
vertex $\Delta (P)$.
\newline
\newline
Now we claim that $_GX_L$ is the only summand of $Res_{G\times L}^{G\times
G}(_G\mathcal{O}Gb_G)$ with vertex containing $\Delta (P)$. In a direct
decomposition of $Res_{G\times N}^{G\times G}(_G\mathcal{O}Gb_G)$ every summand
is either $_GX_N$, has a vertex strictly smaller than $\Delta (P)$ or has a
vertex that is conjugate to $\Delta (P)$ in $(G\times G)$ but not in $(G\times
N)$. So when we restrict down to $(G\times L)$ every summand is either $_GX_L$,
has a vertex strictly smaller than $\Delta (P)$ in size or has a vertex that is
conjugate to $\Delta (P)$ in $(G\times G)$ but not in $(G\times N)$ so certainly
not in $(G\times L)$. So $_GX_L$ is the only summand of $Res_{G\times
L}^{G\times G}(_G\mathcal{O}Gb_G)$ with vertex containing $\Delta (P)$.
\newline
\newline
Let $_GY_L=\mathcal{O}Gb_0U_1^+b_1\dots U_w^+b_w$. Each $b_i$ and each $U_i^+$
is an idempotent in $(\mathcal{O}G)^L$. These idempotents commute with each
other and so their product is an idempotent contained in $(\mathcal{O}G)^L$.
Hence, $_GY_L$ is also a direct summand of $_G\mathcal{O}Gb_L$. We show that it
has as direct summands all summands of $_G\mathcal{O}Gb_L$ with vertex
containing $\Delta (P)$, using the Brauer homomorphism. This tells us
$_GY_L=_GX_L\oplus *$ and consequently we have an $\mathcal{O}$-split
monomorphism $\operatorname{End}_{\mathcal{O}G}(X)\hookrightarrow \operatorname{End}_{\mathcal{O}G}(Y)$. The
calculation goes:

\begin{align*}
Y(\Delta (P))&=\mathcal{O}Gb_0U_1^+b_1\dots U_w^+b_w(\Delta (P))\\
&=kC_G((P))Br_{P}^G(b_0U_1^+b_1\dots U_w^+b_w)\\
&=kC_G((P))Br_{P}^G(b_0)Br_{P}^G(U_1^+)Br_{P}^G(b_1)\dots
Br_{P}^G(U_w^+)Br_{P}^G(b_w)\\
&=kC_G((P))Br_{P}^G(b_0)Br_{P}^G(b_1)\dots Br_{P}^G(b_w)\\
&=kC_G((P))Br_{P}^G(b)\\
&=\mathcal{O}Gb(\Delta (P))\\
\end{align*}

So $Y$ has indeed as summands all summands of $_G\mathcal{O}Gb_L$ with vertex
containing $\Delta (P)$.
\end{proof}

Let $\varphi$ be the character of $KLf$ as a representation of $L$.

\begin{prop}\label{prop:orank}
The $\mathcal{O}$-rank of $\mathcal{O}Nf$ and $\operatorname{End}_{\mathcal{O}G}(Y)$
are both equal to
\newline
$2^ww!\dim_K(KLf)$.
\end{prop}

Before we prove the above proposition we state the following lemma of Chuang and
Kessar\cite[Lemma 4.2]{chukes2002} with out proof.

\begin{lem}\label{lem:aba}
Let $d$ be a positive integer and $\sigma$ a partition equal to its own $d$-core
with a $d$-abacus representation such that the $i$th runner has at least $w-1$
fewer beads than the $(i+1)$th runner for $0\leq i\leq d-2$ and fix this abacus
representation of this core. Let $\lambda$ be a partition with $d$-core $\sigma$
and weight $v\leq w$. Let $\mu$ be a partition such that $\mu_i\leq\lambda_i$
for all $i$ with $d$-core $\sigma$ and weight $v-1$. Then $\mu$ is obtained by
removing a $d$-hook from $\lambda$. If this removal occurs on the $\alpha$th
runner then the complement of the Young diagram of $\mu$ in that of $\lambda$ is
the Young diagram of the hook partition $(\alpha+1,1^{(d-\alpha-1)})$.
\end{lem}

\begin{proof}(of~\ref{prop:orank})
Since $\mathcal{O}$ is a principal ideal domain, proving the proposition
is equivalent to showing that $K\otimes_\mathcal{O}\mathcal{O}Nf$
and $K\otimes_\mathcal{O}\operatorname{End}_{\mathcal{O}G}(Y)$ have the same dimension over
$K$. Now $K\otimes_\mathcal{O}\mathcal{O}Nf\cong KNf\cong Ind_L^N(KLf)$. So
$\dim_K(K\otimes_\mathcal{O}\mathcal{O}Nf)=2^ww!\dim_K(KLf)$. Also
$K\otimes_\mathcal{O}\operatorname{End}_{\mathcal{O}G}(Y)\cong
\operatorname{End}_{KG}(K\otimes_\mathcal{O}Y)$. Additionally we have
$K\otimes_\mathcal{O}Y\cong(K\otimes_\mathcal{O}Y)\otimes_{KLf}KLf$. Let $\varphi$
be the character of $KLf$ as a character of $L$. We will calculate
$R_{G_1,b_1}^{G_0,b_0}\dots R_{G_w,b_w}^{G_{w-1},b_{w-1}}(\varphi)$.
\newline
\newline
Let $\rho^0$ and $\rho^1$ be the two partitions associated to $\rho$ as
in~\ref{not:char}. Our condition~\ref{cond:rho} on $\rho$ demands that $\rho^0$
and $\rho^1$ both satisfy the conditions of $\sigma$ in~\ref{lem:aba} for 
our $d$ and $w$. We will now explain what this means in terms of Harish-Chandra
induction~\ref{sec:har}:
\newline
\newline
Let $0\leq i\leq w-1$ and $\chi_{1,\tau}$ be a unipotent character of
$G_{m-2d(i+1)}(q)$ such that $\tau$ has $e$-core $\rho$. Then:

\begin{gather*}
R_{GL_d(q)_i\times G_{m-2d(i+1)}(q),a_{i+1}\otimes
f_{w-i-1}}^{G_{m-2di}(q),f_{w-i}}(\chi_{1,(\alpha+1,1^{(d-\alpha-1)})}
\otimes\chi_{1,\tau})
\end{gather*}

is equal to the sum of the $\chi_{1,\lambda}$s where $\lambda$ is a partition
obtained from $\tau$ by sliding a bead $1$ place down the $2\alpha$th or
$(2\alpha+1)$th runner in case 1 and the $\alpha$th or
$(\alpha+d)$th runner in cases 2,3 and 4.
\newline
\newline
Let us count the number of ways of sliding single beads down the $l$th runner of
a core $j$ times, so that on the resulting runner the bottom bead has been moved
down $\sigma_1^l$ times, the second bottom bead has been moved down $\sigma_2^l$
times, etc, so that $\sigma_1^l\geq\sigma_2^l\geq\dots$ and
$\sum_i\sigma_i^l=j$. It is equal to the number of ways of writing the numbers
$1,\dots,j$ in the Young diagram of $\sigma_1^l,\sigma_2^l,\dots$ so that
numbers increase across rows and down columns, that is, the degree of the
character $\zeta^{\sigma^l}$ of the symmetric group $S_j$.
\newline
\newline
The irresducible characters in the block $KLf$ are of the form
$\chi_{s_1,\lambda_1}\otimes\dots\otimes\chi_{s_w,\lambda_w}\otimes\chi_{s,\rho}
$ where either $s_i$ is $1$ and $\lambda_i$ is an $d$-hook partition or $s_i$ is
a non-trivial $p$-element of $GL_d(q)$ and $\lambda_i$ is the partition
$(1)$~\ref{sec:charblock}. Note that $s$ is always $1$ or something whose image
under $\pi$ is $1$ in cases 2(b) and 4(b). In other words,
in the latter case, $s$ is just a $p$-element of the underlying field in the
special Clifford group.
\newline
\newline
We can now describe $R_{G_1,b_1}^{G_0,b_0}\dots
R_{G_w,b_w}^{G_{w-1},b_{w-1}}(\chi_{s_1,{\lambda_1}}\otimes\dots
\otimes\chi_{s_w,{\lambda_w}}\otimes\chi_{s,\rho})$.
Suppose that the $s_i$s are grouped together with like elements such that
$s_1$,\dots,$s_{r_0}$ are all equal to $1$ and the remaining elements are
grouped together into conjugacy classes as follows:

\begin{align*}
&s_{r_0+1}\sim\dots\sim s_{r_0+\alpha_1}\sim t_1,\\
&s_{r_0+\alpha_1+1}\sim\dots\sim s_{r_0+\alpha_1+\beta_1}\sim \overline{t_1},\\
&s_{r_0+\alpha_1+\beta_1+1}\sim\dots\sim s_{r_0+\alpha_1+\beta_1+\alpha_2}\sim
t_2,\\
&s_{r_0+\alpha_1+\beta_1+\alpha_2+1}\sim\dots\sim
s_{r_0+\alpha_1+\beta_1+\alpha_2+\beta_2}\sim\overline{t_2},\\
&\dots
\end{align*}

where $\overline{t}$ means $t^{-1}$ (or $t^{-q_0}$ in case 1).
Additionally set $r_i=\alpha_i+\beta_i$ giving $\sum_ir_i=w$ and:

\begin{align*}
&\lambda_1=\dots=\lambda_{l_0}=(1^d),\\
&\lambda_{l_0+1}=\dots=\lambda_{l_0+l_1}=(2,1^{d-2}),\\
&\dots,\\
&\lambda_{l_0+\dots+l_{d-2}+1}=\dots=\lambda_{l_0+\dots+l_{d-1}}=(d)
\end{align*}

where  $\sum_il_i=r_0$.
\newline
\newline
Now we have an expression for $R_{G_1,b_1}^{G_0,b_0}\dots
R_{G_w,b_w}^{G_{w-1},b_{w-1}}(\chi_{s_1,{\lambda_1}}\otimes\dots\otimes\chi_{s_w
,{\lambda_w}}\otimes\chi_{s,\rho})$:

\begin{align*}
\sum_{\begin{smallmatrix}
l_0+\dots+l_d=r_0\\
|\sigma^i|+|\tau^i|=l_i, |\nu^i|=r_i
\end{smallmatrix}}&{l_0\choose|\sigma^0|}\dim\zeta^{\sigma^0}\dim\zeta^{\tau^0}
\dots{l_{d-1}\choose|\sigma^{d-1}|}\dim\zeta^{\sigma^{d-1}}\dim\zeta^{\tau^{d-1}}
\\
&\dim\zeta^{\nu^1}\dim\zeta^{\nu^2}\dots\\
&\chi(s\times s_1\times \overline{s_1}\times\dots\times s_w\times
\overline{s_w},\mu)
\end{align*}

Where $\mu_{X-1}$ is the partition/symbol whose $e$-core is $\rho$ and whose
$e$-quotients are $[\sigma^0,\dots,\sigma^{d-1}]$ and
$[\tau^0,\dots,\tau^{d-1}]$ (or just
$[\sigma^0,\tau^0,\dots,\sigma^{d-1},\tau^{d-1}]$ in case 1)
with respect to the $2d$-abacus representation of $\rho$ already
fixed~\ref{cond:rho} and $\mu_{\psi_i}=\nu^i$ where $\psi_i$ is the minimal
polynomial of $t_i\times\overline{t_i}$.
\newline
\newline
Now we can permute the $\lambda_i$s and still get the same character of $G$ when
we do Harish-Chandra induction. There are exactly $w!/(l_0!l_1!\dots
l_{d-1}!\alpha_1!\beta_1!\alpha_2!\beta_2!\dots)$ such permutations. Also each
irreducible character appears in $\varphi$ with multiplicity equal to the
dimension of said character. So we have the following expression for
$R_{G_1,b_1}^{G_0,b_0}\dots R_{G_w,b_w}^{G_{w-1},b_{w-1}}(\varphi)$:

\begin{align*}
\sum_{\begin{smallmatrix}
l_0+\dots+l_{d-1}+r_1+r_2+\dots=w \\
\alpha_i+\beta_i=r_i,|\sigma^i|+|\tau^i|=l_i, \kappa
\end{smallmatrix}}&\frac{w!}{l_0!l_1!\dots
l_{d-1}!\alpha_1!\beta_1!\alpha_2!\beta_2!\dots}\\
&\dim(\chi_{s_1,{\lambda_1}}\otimes\dots\otimes\chi_{s_w,{\lambda_w}}\otimes\chi_
{s,\rho})\\
&{l_0\choose|\sigma^0|}\dim\zeta^{\sigma^0}\dim\zeta^{\tau^0}\dots{l_{d-1}
\choose|\sigma^{d-1}|}\dim\zeta^{\sigma^{d-1}}\dim\zeta^{\tau^{d-1}}\\
&\dim\zeta^{\nu^1}\dim\zeta^{\nu^2}\dots\\
&\chi(s\times s_1\times \overline{s_1}\times\dots\times s_w\times
\overline{s_w},\mu)
\end{align*}

Where $s$ is the scalar matrix $\kappa I$. So $\kappa$ runs over
$(\mathbb{F}^\times_q)_p$ in cases 2(b) and 4(b) and just $1$
otherwise. So we get:

\begin{align*}
\sum_{\begin{smallmatrix}
l_0+\dots+l_{d-1}+r_1+r_2+\dots=w \\
|\sigma^i|+|\tau^i|=l_i, \kappa
\end{smallmatrix}}&\frac{w!}{l_0!l_1!\dots l_{d-1}!r_1!r_2!\dots}\\
&{l_0\choose|\sigma^0|}\dim\zeta^{\sigma^0}\dim\zeta^{\tau^0}\dots{l_{d-1}
\choose|\sigma^{d-1}|}\dim\zeta^{\sigma^{d-1}}\dim\zeta^{\tau^{d-1}}\\
&\dim\zeta^{\nu^1}\dim\zeta^{\nu^2}\dots\\
&\chi(s\times s_1\times \overline{s_1}\times\dots\times s_w\times
\overline{s_w},\mu)\\
&[\sum_{\alpha_i+\beta_i=r_i}{r_1\choose\alpha_1}{r_2\choose\alpha_2}\dots
\dim(\chi_{s_1,{\lambda_1}}\otimes\dots\otimes\chi_{s_w,{\lambda_w}}\otimes\chi_{
s,\rho})]
\end{align*}

So the dimension of $\operatorname{End}_{KG}((K\otimes_\mathcal{O}Y)\otimes_KKLf)$ over $K$ is:

\begin{align*}
\sum_{\begin{smallmatrix}
l_0+\dots+r_1+\dots=w \\
|\sigma^i|+|\tau^i|=l_i, \kappa
\end{smallmatrix}}&
(\frac{w!}{l_0!l_1!\dots l_{d-1}!r_1!r_2!\dots})^2\\
&{l_0\choose|\sigma^0|}^2(\dim\zeta^{\sigma^0})^2(\dim\zeta^{\tau^0})^2\dots{l_{
d-1}\choose|\sigma^{d-1}|}^2(\dim\zeta^{\sigma^{d-1}})^2(\dim\zeta^{\tau^{d-1}}
)^2\\
&(\dim\zeta^{\nu^1})^2(\dim\zeta^{\nu^2})^2\dots\\
&[\sum_{\alpha_i+\beta_i=r_i}{r_1\choose\alpha_1}{r_2\choose\alpha_2}\dots
\dim(\chi_{s_1,{\lambda_1}}\otimes\dots\otimes\chi_{s_w,{\lambda_w}}\otimes\chi_{
s,\rho})]^2\\
\end{align*}

\begin{align*}
=\sum_{\begin{smallmatrix}
l_0+\dots+r_1+\dots=w \\
\sigma_i+\tau_i=l_i, \kappa
\end{smallmatrix}}&
(\frac{w!}{l_0!l_1!\dots
l_{d-1}!r_1!r_2!\dots})^2{l_0\choose\sigma_0}^2\dots{l_{d-1}\choose\sigma_{d-1}}
^2\\
&[\sum_{\sigma^i\vdash\sigma_i,\tau^i\vdash\tau_i}(\dim\zeta^{\sigma^0}
)^2(\dim\zeta^{\tau^0})^2\dots(\dim\zeta^{\sigma^{d-1}})^2(\dim\zeta^{\tau^{d-1}}
)^2]\\
&(\dim\zeta^{\nu^1})^2(\dim\zeta^{\nu^2})^2\dots\\
&[\sum_{\alpha_i+\beta_i=r_i}{r_1\choose\alpha_1}{r_2\choose\alpha_2}\dots
\dim(\chi_{s_1,\lambda_1})\dots \dim(\chi_{s_w,\lambda_w})\dim(\chi_{s,\rho})]^2
\end{align*}

Now using the fact that $\sum_{\sigma\vdash h}\dim(\zeta^\sigma)^2=h!$ we get:

\begin{align*}
\sum_{\begin{smallmatrix}
l_0+\dots+r_1+\dots=w \\
\sigma_i+\tau_i=l_i, \kappa
\end{smallmatrix}}&(\frac{w!}{l_0!l_1!\dots
l_{d-1}!r_1!r_2!\dots})^2{l_0\choose\sigma_0}^2\dots{l_{d-1}\choose\sigma_{d-1}}
^2\\
&\sigma_0!\tau_0!\dots\sigma_{d-1}!\tau_{d-1}!r_1!r_2!\dots\\
&[\sum_{\alpha_i+\beta_i=r_i}{r_1\choose\alpha_1}{r_2\choose\alpha_2}\dots
\dim(\chi_{s_1,\lambda_1})\dots \dim(\chi_{s_w,\lambda_w})\dim(\chi_{s,\rho})]^2\\
\end{align*}

Now $\dim(\chi_{t,(1)})=\dim(\chi_{\overline{t},(1)})$~\ref{sec:dim}. This means
that for fixed $(r_1,r_2,\dots)$ the choice of the $\alpha_i$s does not affect
$\dim(\chi_{s_1,\lambda_1})\dots \dim(\chi_{s_w,\lambda_w})\dim(\chi_{s,\rho})$.
Using this and the fact that $\sum_{i=0}^r {r\choose i}=2^r$ we get:

\begin{align*}
w!\sum_{\begin{smallmatrix}
l_0+\dots+r_1+\dots=w \\
\sigma_i+\tau_i=l_i, \kappa
\end{smallmatrix}}&\frac{w!}{l_0!l_1!\dots
l_{d-1}!r_1!r_2!\dots}{l_0\choose\sigma_0}\dots{l_{d-1}\choose\sigma_{d-1}}\\
&2^{\sum
r_i}[\sum_{\alpha_i+\beta_i=r_i}{r_1\choose\alpha_1}{r_2\choose\alpha_2}\dots
\dim(\chi_{s_1,\lambda_1}\otimes\dots\otimes\chi_{s_w,\lambda_w}\otimes\chi_{s,
\rho})^2]\\
\end{align*}

\begin{align*}
=w!\sum_{\begin{smallmatrix}
l_0+\dots+r_1+\dots=w \\
\kappa
\end{smallmatrix}}&\frac{w!}{l_0!l_1!\dots l_{d-1}!r_1!r_2!\dots}\\
&2^{\sum r_i}2^{\sum
l_i}(\sum_{\alpha_i+\beta_i=r_i}{r_1\choose\alpha_1}\dots
\dim(\chi_{s_1,\lambda_1}\otimes\dots\otimes\chi_{s_w,\lambda_w}\otimes\chi_{s,
\rho})^2)\\
\end{align*}

\begin{align*}
=2^ww!\sum_{\begin{smallmatrix}
l_0+\dots+l_{d-1}+\alpha_1+\beta_1+\alpha_2+\beta_2\dots=w \\
\kappa
\end{smallmatrix}}&\frac{w!}{l_0!l_1!\dots
l_{d-1}!\alpha_1!\beta_1!\alpha_2!\beta_2!\dots}\\
&\dim(\chi_{s_1,\lambda_1}\otimes\dots\otimes\chi_{s_w,\lambda_w}\otimes\chi_{s,
\rho})^2\\
\end{align*}

\begin{align*}
=2^ww!\dim_K(KLf)
\end{align*}
\end{proof}

\begin{proof}(of~\ref{thm:nondegen})
By~\ref{prop:orank} $\mathcal{O}Nf$ and $\operatorname{End}_{\mathcal{O}G}(Y)$ have the same
$\mathcal{O}$-rank. As a consequence we have that all the monomorphisms
in~\ref{prop:spilt} become isomorphisms. Therefore, since $X$ is a progenerator
as a left $\mathcal{O}Gb$-module, $_{\mathcal{O}Gb}X_{\mathcal{O}Nf}$ induces
a Morita equivalence between $\mathcal{O}Gb$ and $\mathcal{O}Nf$.
\end{proof}

\subsection{Degenerate Case}

We now assume that $\rho$ is degenerate and restrict out attention to case 4(b).
Again we will later prove the corresponding theorem for case 4(a) using~\ref{lem:so2n}.

\begin{thm}\label{thm:degen}
$\mathcal{O}N(f+f')$ is a block of $N$ and is Morita equivalent to
$\mathcal{O}Gb$.
\end{thm}

The proof of this theorem will closely resemble that of~\ref{lem:so2n}.
We will need all the corresponding lemmas first.

\begin{lem}$ $
\begin{enumerate}
\item $P$ is defect group for $\mathcal{O}G_ib_i$ for $(0\leq i\leq w-1)$ and
also for $\mathcal{O}Lf$ and $\mathcal{O}Lf'$.
\item $Br_P^G(b_i)=Br_P^G(f)=Br_P^G(f')=a'_1\otimes\dots\otimes a'_w\otimes
(f_0+f'_0)$ where $a'_i$ is the principal block of $C_{GL_d(q)_i}(P_i)$ for
$(0\leq i\leq w-1)$. In addition we have $Br_P^G(U_i^+)=1$.
\item $N$ stabilizes $(f+f')$, $\mathcal{O}Nf$ and $\mathcal{O}Nf'$ are both
indecomposable with vertex $\Delta(P)$ and $\mathcal{O}N(f+f')$ is a block of
$N$.
\item $\mathcal{O}Gb$ and $\mathcal{O}N(f+f')$ both have defect group $P$ and
are Brauer correspondents.
\end{enumerate}
\end{lem}

\begin{proof}$ $
\begin{enumerate}
\item $P_i$ is defect group for $\mathcal{O}GL_d(q)_ia_i$ and
$(P_{i+1}\times\dots\times P_w).Z_p$ is a defect group for
$\mathcal{O}G_{m-di}(q)f_{w-i}$. Similarly for $\mathcal{O}Lf$ and
$\mathcal{O}Lf'$.
\item Identical to~\ref{lem:omn} part (2).
\item As with the non-degenerate case the only thing to check for the first part
is that $CO^+_{m-2dw}(q)$ stabilizes $(f+f')$. This is again clear as in the
non-degenerate case. Also as in the non-degenerate case we have both
$\mathcal{O}Nf$ and $\mathcal{O}Nf'$ are indecomposable as $\mathcal{O}(N\times
L)$-modules with vertex $\Delta(P)$. Note that $\mathcal{O}Nf$ and
$\mathcal{O}Nf'$ lie in different blocks of $\mathcal{O}(N\times L)$. This
implies that if $\mathcal{O}N(f+f')$ were decomposable as a $\mathcal{O}(N\times
N)$-module we would have to get the same decomposition
$\mathcal{O}N(f+f')\cong\mathcal{O}Nf\oplus\mathcal{O}Nf'$. However $N_G(P)\leq
N$ is transitive on the blocks of $C_G(P)$ appearing in the image of $b$ under
$Br_P^G$. This means $f_0$ is conjugate to $f'_0$ and hence $f$ to $f'$. Thus
$\mathcal{O}N(f+f')$ is one block of $N$.
\item $P$ is a defect group for $\mathcal{O}Gb$ by part (1). As in the
non-degenerate case every $p$-subgroup of $N$ lies in $L$ so
$\mathcal{O}N(f+f')$ has defect group $P$. Also $Br_P^N(f+f')=Br_P^G(b)$ by (2)
and so $\mathcal{O}N(f+f')$ is the Brauer correspondent of $\mathcal{O}Gb$ in
$N$.
\end{enumerate}
\end{proof}

We again let $X$ be the Green correspondent of $\mathcal{O}Gb$ in $G\times N$.
This time we set $Y= _GY_L=\mathcal{O}Gb_0U_1^+b_1\dots U_w^+(b_w+b_w')$. Recall
that $f=b_w$ and $f'=b_w'$.

\begin{prop}\label{prop:split2}
There is a sequence of $\mathcal{O}$-split monomorphisms of algebras

\begin{gather*}
\mathcal{O}N(f+f')\hookrightarrow \operatorname{End}_{\mathcal{O}G}(X)\hookrightarrow
\operatorname{End}_{\mathcal{O}G}(Y)
\end{gather*}

Also the left $\mathcal{O}Gb$-module $X$ is a progenerator for $\mathcal{O}Gb$.
\end{prop}

During the proof we will refer back to the corresponding proposition for the
non-degenerate case~\ref{prop:spilt}.

\begin{proof}
As in the non-degenerate case $X$ is clearly a progenerator for $\mathcal{O}Gb$.
It is also clear that we have the $\mathcal{O}$-split monomorphism
$\mathcal{O}N(f+f')\hookrightarrow \operatorname{End}_{\mathcal{O}G}(X)$.
\newline
\newline
This time $_GX_L$ is not indecomposable. In instead we have that it is the
direct sum of $2$ indecomposable modules $Xf$ and $Xf'$ each with vertex
$\Delta(P)$. Next, as with the non-degenerate case, we have that $_GX_L$ is the
direct sum of indecomposable modules all with vertex $\Delta(P)$. Finally we see
that $_GX_Lf=_GX_N\otimes_{\mathcal{O}N}\mathcal{O}Nf_L$ is a direct summand of
$Ind_{N\times L}^{G\times L}(\mathcal{O}Nf)$ which by Green correspondence has
exactly one summand with vertex $\Delta (P)$. So $_GX_Lf$ is indecomposable with
vertex $\Delta (P)$. Similarly for $_GX_Lf'$.
\newline
\newline
The proof that we have an $\mathcal{O}$-split monomorphism
$\operatorname{End}_{\mathcal{O}G}(X)\hookrightarrow \operatorname{End}_{\mathcal{O}G}(Y)$ is again as in the
non-degenerate case with $b_w$ replaced by $b_w+b_w'$.
\end{proof}

Let $\varphi$ be the character of $KL(f+f')$ as a representation of $L$.

\begin{prop}\label{prop:orank2}
The $\mathcal{O}$-rank of $\mathcal{O}N(f+f')$ and $\operatorname{End}_{\mathcal{O}G}(Y)$
are both equal to
\newline
$2^ww!\dim_K(KLf)$.
\end{prop}

\begin{proof}
We obtain a combinatorial rule for Harish-Chandra induction in this case by
comparing with the non-degenerate case where everything is calculated in the Weyl
group of type $B$. We first want to calculate

\begin{gather*}
R_{G_{w-v+1},b_{w-v+1}}^{G_{w-v},b_{w-v}}\dots
R_{G_w,(b_w+b_w')}^{G_{w-1},b_{w-1}}(\chi_{1,{\lambda_1}}\otimes\dots\otimes\chi_{1,
{\lambda_v}}\otimes(\chi_{s,\rho}+\chi_{s,\rho}'))
\end{gather*}

where $v\leq w$ and $s$ is a $p$-element of the underlying field $\mathbb{F}_q$.
\newline
\newline
Let $\tau=\{X,Y\}$ be a non-degenerate symbol of weight $w-v$. Then the
multiplicity of $\chi_{s,\tau}$ is just the number of ways of obtaining $\tau$
from $\rho$ by sliding beads down the appropriate runners. (Appropriate means
the runners determined by the $\lambda_i$s. Also obtaining $\tau$ from $\rho$
means obtaining $\tau$ so that $X$ is on the left or the right.)
\newline
\newline
Now let $\tau=\{X,X\}$ be a degenerate symbol of weight $w-v$. Then again the
multiplicity of $\chi_{s,\tau}$ is just the number of ways of obtaining $\tau$
from $\rho$ by sliding beads down the appropriate runners. The same is true of
$\chi_{s,\tau}'$.
\newline
\newline
Adopting the same notation as in the non-degenerate case we let $\varphi$ be the
character of $KL(f+f')$. We have the following expression for
$R_{G_1,b_1}^{G_0,b_0}\dots R_{G_w,(b_w+b_w')}^{G_{w-1},b_{w-1}}(\varphi)$:

\begin{align*}
\sum_{\begin{smallmatrix}
l_0+\dots+l_{d-1}+r_1+r_2+\dots=w \\
\alpha_i+\beta_i=r_i,|\sigma^i|+|\tau^i|=l_i, \kappa\\
\end{smallmatrix}}&\frac{w!}{l_0!l_1!\dots
l_{d-1}!\alpha_1!\beta_1!\alpha_2!\beta_2!\dots}\\
&\dim(\chi_{s_1,{\lambda_1}}\otimes\dots\otimes\chi_{s_w,{\lambda_w}}\otimes\chi_
{s,\rho})\\
&{l_0\choose|\sigma^0|}\dim\zeta^{\sigma^0}\dim\zeta^{\tau^0}\dots{l_{d-1}
\choose|\sigma^{d-1}|}\dim\zeta^{\sigma^{d-1}}\dim\zeta^{\tau^{d-1}}\\
&\dim\zeta^{\nu^1}\dim\zeta^{\nu^2}\dots\\
&\chi(s\times s_1\times \overline{s_1}\times\dots\times s_w\times
\overline{s_w},\mu)
\end{align*}

Compare with the non-degenerate case. $\chi(s\times s_1\times
\overline{s_1}\times\dots\times s_w\times \overline{s_w},\mu)$ means the sum
of both characters when $\mu_{X-1}$ is degenerate. Note that
$\dim(\chi_{s,\rho})=\dim(\chi_{s,\rho}')$ (see~\ref{sec:dim}).

\begin{align*}
=\sum_{\begin{smallmatrix}
l_0+\dots+l_{d-1}+r_1+r_2+\dots=w \\
|\sigma^i|+|\tau^i|=l_i, \kappa\\
\end{smallmatrix}}&\frac{w!}{l_0!l_1!\dots l_{d-1}!r_1!r_2!\dots}\\
&{l_0\choose|\sigma^0|}\dim\zeta^{\sigma^0}\dim\zeta^{\tau^0}\dots{l_{d-1}
\choose|\sigma^{d-1}|}\dim\zeta^{\sigma^{d-1}}\dim\zeta^{\tau^{d-1}}\\
&\dim\zeta^{\nu^1}\dim\zeta^{\nu^2}\dots\\
&\chi(s\times s_1\times \overline{s_1}\times\dots\times s_w\times
\overline{s_w},\mu)\\
&[\sum_{\alpha_i+\beta_i=r_i}{r_1\choose\alpha_1}{r_2\choose\alpha_2}\dots
\dim(\chi_{s_1,{\lambda_1}}\otimes\dots\otimes\chi_{s_w,{\lambda_w}}\otimes\chi_{
s,\rho})]
\end{align*}

Note that if we swap all the $\sigma^i$s and $\tau^i$s over then we get the same
character. Recalling also that $\chi(s\times s_1\times
\overline{s_1}\times\dots\times s_w\times \overline{s_w},\mu)$ means the sum of both
characters when $\mu_{X-1}$ is degenerate the dimension of
$\operatorname{End}_{KG}((K\otimes_\mathcal{O}Y)\otimes_KKLf)$ over $K$ is:

\begin{align*}
2\sum_{\begin{smallmatrix}
l_0+\dots+r_1+\dots=w \\
|\sigma^i|+|\tau^i|=l_i, \kappa\\
\end{smallmatrix}}&(\frac{w!}{l_0!l_1!\dots l_{d-1}!r_1!r_2!\dots})^2\\
&{l_0\choose|\sigma^0|}^2(\dim\zeta^{\sigma^0})^2(\dim\zeta^{\tau^0})^2\dots{l_{
d-1}\choose|\sigma^{d-1}|}^2(\dim\zeta^{\sigma^{d-1}})^2(\dim\zeta^{\tau^{d-1}}
)^2\\
&(\dim\zeta^{\nu^1})^2(\dim\zeta^{\nu^2})^2\dots\\
&[\sum_{\alpha_i+\beta_i=r_i}{r_1\choose\alpha_1}{r_2\choose\alpha_2}\dots
\dim(\chi_{s_1,{\lambda_1}}\otimes\dots\otimes\chi_{s_w,{\lambda_w}}\otimes\chi_{
s,\rho})]^2\\
\end{align*}

Following all the step through as before we get:

\begin{align*}
2^{w+1}w!\sum_{\begin{smallmatrix}
l_0+\dots+\alpha_1+\beta_1+\alpha_2+\beta_2\dots=w \\
\kappa\\
\end{smallmatrix}}&\frac{w!}{l_0!l_1!\dots
l_{d-1}!\alpha_1!\beta_1!\alpha_2!\beta_2!\dots}\\
&\dim(\chi_{s_1,\lambda_1}\otimes\dots\otimes\chi_{s_w,\lambda_w}\otimes\chi_{s,
\rho})^2\\
\end{align*}

Recall that
$\dim(\chi_{s_1,{\lambda_1}}\otimes\dots\otimes\chi_{s_w,{\lambda_w}}\otimes\chi_
{s,\rho})=\dim(\chi_{s_1,{\lambda_1}}\otimes\dots\otimes\chi_{s_w,{\lambda_w}}
\otimes\chi_{s,\rho}')$ to obtain:

\begin{align*}
=2^ww!\dim_K(KLf)
\end{align*}
\end{proof}

\begin{proof}(of~\ref{thm:degen})
Again~\ref{prop:orank2} and~\ref{prop:split2} give us the relevant information to
obtain that $_{\mathcal{O}Gb}X_{\mathcal{O}N(f+f')}$ induces
a Morita equivalence between $\mathcal{O}Gb$ and $\mathcal{O}Nf$.
\end{proof}

\subsection{\texorpdfstring{$Sp_{2n}(q)$}{TEXT} and \texorpdfstring{$SO_{2n}^\pm(q)$}{TEXT}}

We are now is a position to prove~\ref{thm:nondegen} and~\ref{thm:degen}
our main theorem for cases 2(a) and 4(a).

\begin{cor}\label{cor:so2n}$ $
\begin{enumerate}
\item If $\rho$ is non-degenerate then $\mathcal{O}Nf$ is a block of $N$ and is
Morita equivalent to $\mathcal{O}Gb$.
\item If $\rho$ is degenerate then $\mathcal{O}N(f+f')$ is a block of $N$ and is
Morita equivalent to $\mathcal{O}Gb$.
\end{enumerate}
\end{cor}

Of course part (2) only happens in case 4(a).
\newline
We will use the proof of the corresponding theorems for cases 2(b) and
4(b) (see~\ref{thm:nondegen} and~\ref{thm:degen}) as well
as~\ref{lem:so2n}.

\begin{proof}
Adopting the notation of~\ref{lem:so2n} we have the correspondence between 
characters of $\mathcal{O}\hat{G}j$ and 
$\mathcal{O}Gi$ given by
restriction but with each character of $\mathcal{O}Gi$ appearing as
a restriction of $(q-1)_p$ distinct characters of
$\mathcal{O}\hat{G}j$.
\newline
Given that this correspondence commutes with
Harish-Chandra induction we have that the proofs of~\ref{thm:nondegen}
and~\ref{thm:degen} run through for $G$.
Note that the character sums have to be divided by $(q-1)_p$.
\end{proof}

All that remains to prove theorem~\ref{thm:intro}
is the following lemma. We continue with the notation from this section as well
as that from~\ref{thm:intro}. We will consider all 4 cases.

\begin{lem}\label{lem:twist}$ $
\begin{enumerate}
\item If $\rho$ is non-degenerate then
$\mathcal{O}Nf$ is Morita equivalent to the principal
block of $\mathcal{O}((GL_d(q).2\wr S_w)\times Z_p)$.
\item If $\rho$ is degenerate then $\mathcal{O}N(f+f')$ is Morita equivalent to the principal block of
$\mathcal{O}(M\times Z_p)$.
\end{enumerate}
\end{lem}

\begin{proof}
We prove for cases 1,2(a),3 and 4(a). The
statements for cases 2(b) and 4(b) will then follow from the corresponding statements
for cases 2(a) and 4(a) and~\ref{lem:so2n}.

\begin{enumerate}
\item This is clear for cases 1 and 2(a). For cases 3 and 4(a) we take the definition of $\widetilde{G}$ from the
beginning of the section. First consider the unique character of $\mathcal{O}G_{m-2dw}(q)f_0$. Since $\mathcal{O}Nf$
is a block this character is invariant under conjugation by $\widetilde{G}_{m-2dw}(q)$ as $\mathcal{O}G_{m-2dw}(q)f_0$ is.
Therefore this character induces to 2 different characters of $\widetilde{G}_{m-2dw}(q)$ and hence
$\mathcal{O}\widetilde{G}_{m-2dw}(q)f_0$ is not a block and is in fact the direct sum of 2 blocks,
$\mathcal{O}\widetilde{G}_{m-2dw}(q)f'_0$ and $\mathcal{O}\widetilde{G}_{m-2dw}(q)f''_0$ both Morita equivalent to
$\mathcal{O}G_{m-2dw}(q)f_0$ due to\cite[Proposition 6]{cabeng1993}.
\newline
\newline
$\mathcal{O}Nf$ is Morita equivalent to the block of $\mathcal{O}((GL_d(q).2\wr S_w)\times\widetilde{G}_{m-2dw}(q))$
with block idempotent $a_1\otimes\dots\otimes a_w\otimes f'_0$ by\cite[Proposition 6]{cabeng1993} which is in
turn clearly Morita equivalent to the principal block of $\mathcal{O}(GL_d(q).2\wr S_w)$.
\item Let $a'$ be the principal block idempotent of $\mathcal{O}M$. Then
$\mathcal{O}Ma'$ is Morita equivalent to
$\mathcal{O}Ma'\otimes_\mathcal{O}\mathcal{O}G_{m-2dw}f_0$ which is in turn
Morita equivalent to $\mathcal{O}N(f+f')$.
\end{enumerate}
\end{proof}

\section{Application to Brou\'{e}'s Conjecture}

We are now in a position to show that $\mathcal{O}Gb$ is derived equivalent to its
Brauer correspondent in $N_G(P)$. We will use $h$ to denote the corresponding block idempotent of $\mathcal{O}N_G(P)$.

\begin{cor}\label{cor:bottom}
$\mathcal{O}Gb$ is derived equivalent to its Brauer correspondent in $N_G(P)$.
\end{cor}

We will need a couple of lemmas to prove the above corollary but first we define $M'$ analogously to how we defined $M$ in the introduction.
\newline
\newline
We have a natural homomorphism $N_{GL_d(q).2}(R)\rightarrow\{\pm1\}$
with kernel $N_{GL_d(q)}(R)$. This extends to a map $N_{GL_d(q).2}(R)\wr
S_w\rightarrow\{\pm1\}$ and we use $M'$ to denote the kernel of this map.

\begin{lem}\label{lem:twist2}$ $
\begin{enumerate}
\item If $\rho$ is non-degenerate then
$\mathcal{O}N_G(P)h$ is Morita equivalent to the principal
block of $\mathcal{O}((N_{GL_d(q).2}(R)\wr S_w)\times Z_p)$.
\item If $\rho$ is degenerate then $\mathcal{O}N_G(P)h$ is Morita equivalent to the principal block of
$\mathcal{O}(M'\times Z_p)$.
\end{enumerate}
\end{lem}

\begin{proof}
See proof of~\ref{lem:twist}.
\end{proof}

We will need the following theorem due to A. Marcus\cite[Theorem 3.2(b)]{marcus1994}.

\begin{thm}\label{thm:deriv}
Let $X$, $\widetilde{X}$, $Y$ and $\widetilde{Y}$ be finite groups with
$X\vartriangleleft \widetilde{X}$, $Y\vartriangleleft \widetilde{Y}$, $X\leq Y$ and
$\widetilde{X}\leq\widetilde{Y}$. We also require that $\widetilde{X}\cap Y=X$ and that the
natural homomorphism $\widetilde{X}/X\rightarrow\widetilde{Y}/Y$ is in fact an
isomorphism of $p'$-groups. Now let $i$ (respectively $j$) be a block idempotent
of $\mathcal{O}X$ (respectively $\mathcal{O}Y$) which is fixed by conjugation by
$\widetilde{X}$ (respectively $\widetilde{Y}$) and
$(\mathcal{C},\mathcal{C}^*)$ a pair of complexes giving a derived equivalence
between $\mathcal{O}Xi$ and $\mathcal{O}Yj$. Now consider $A$ the subalgebra of
$\mathcal{O}\widetilde{X}i\otimes_{\mathcal{O}}(\mathcal{O}\widetilde{Y}j)^{op}$
generated by $\mathcal{O}Xi\otimes_{\mathcal{O}}(\mathcal{O}Yj)^{op}$ and
$xi\otimes x^{-1}j$ for all $x\in \widetilde{X}$. Suppose that $\mathcal{C}$ extends
to complex of $A$-modules, then
$(\mathcal{O}\widetilde{X}i\otimes_{\mathcal{O}Xi}\mathcal{C},\mathcal{O}\widetilde{Y}
j\otimes_{\mathcal{O}Yi}\mathcal{C}^*)$ is a pair of complexes giving a derived
equivalence between $\mathcal{O}\widetilde{X}i$ and $\mathcal{O}\widetilde{Y}j$.
\end{thm}

If $\mathcal{C}$ extends to complex of $A$-modules we will say that
$\mathcal{C}$ has a consistent diagonal action of $\widetilde{X}$.
\newline
\newline
We are now in a position to prove corollary~\ref{cor:bottom}.

\begin{proof}~\ref{cor:bottom}
We already have that $\mathcal{O}Gb$ is Morita equvalent to $\mathcal{O}Nf$~\ref{thm:nondegen}
(or $\mathcal{O}N(f+f')$ in the degenerate case~\ref{thm:degen}) which is in turn Morita equivalent to
the principal block of $\mathcal{O}((GL_d(q).2\wr S_w)\times Z_p)$ (respectively $\mathcal{O}(M\times Z_p)$)
~\ref{lem:twist}. On the other hand $\mathcal{O}N_G(P)h$ is Morita equivalent to the principal block of
$\mathcal{O}((N_{GL_d(q).2}(R)\wr S_w)\times Z_p)$ (respectively $\mathcal{O}(M'\times Z_p)$)~\ref{lem:twist2}.
Therefore all that remains to prove~\ref{cor:bottom} is that the principal blocks of $\mathcal{O}(GL_d(q).2\wr S_w)$
and $\mathcal{O}(N_{GL_d(q).2}(R)\wr S_w)$ (respectively $\mathcal{O}M$ and $\mathcal{O}M'$) are derived equvalent.
\newline
\newline
$R$ is cyclic so the principal block of $\mathcal{O}GL_d(q).2$ is derived
equivalent to the principal block of $\mathcal{O}N_{GL_d(q).2}(R)$. Then by
\cite[Theorem 4.3(b)]{marcus1994} the principal block of
$\mathcal{O}(GL_d(q).2\wr S_w)$ is derived equivalent to the principal block of
$\mathcal{O}(N_{GL_d(q).2}(R)\wr S_w)$.
\newline
\newline
By\cite[Example 5.5]{marcus1994} there exists a pair of complexes
$(\mathcal{C},\mathcal{C}^*)$ that induce a derived equivalence between the
principal blocks of $\mathcal{O}N_{GL_d(q)}(R)$ and $\mathcal{O}GL_d(q)$ such
that $\mathcal{C}$ has a consistent diagonal action of $N_{GL_d(q).2}(R)$.
\newline
\newline
By\cite[Theorem 4.3(b)]{marcus1994} we can construct a derived equivalence
between the principal blocks of $\mathcal{O}(N_{GL_d(q)}(R)\wr S_w)$ and
$\mathcal{O}(GL_d(q)\wr S_w)$. Let $(\mathcal{D},\mathcal{D}^*)$ be the pair of complexes giving this equivalence.
Note that from the construction of $(\mathcal{D},\mathcal{D}^*)$ and the
previous paragraph that $\mathcal{D}$ has a consistent diagonal action of
$(N_{GL_d(q).2}(R))\wr S_w$ so certainly of $M'$. Therefore by~\ref{thm:deriv} we can construct a derived equivalence
between the principal blocks of $\mathcal{O}M$ and $\mathcal{O}M'$.
\end{proof}

Note that lemma~\ref{lem:twist2} still holds without our condition~\ref{cond:rho}. Therefore we have the following
corollary for all four cases.

\begin{cor}\label{cor:norm}
Let $B$ be a unipotent block of $\mathcal{O}G_{m'}(q)$ of weight $w$ with abelian defect group $P'$. If $B'$ is the
Brauer correspondent of $B$ in $N_{G_{m'}(q)}(P')$ then the Morita equivalence class of $B'$ depends only on $w$ and
whether $B$ is degenerate or not.
\end{cor}

The above corollary is an analogy of part (3) of the proof of Brou\'{e}'s conjecture for the symmetric group given in the introduction.

\begin{rem}
We note that a corresponding theorem to~\ref{thm:intro} for the case of $q$ being even has yet to be proven for all
but the unitary group.
\end{rem}

\nocite{turner2002, churou2008, chukes2002, hiskes2000, james1978, fonsri1982,
fonsri1989, marcus1994, kulsha1981, cabeng1993, broue1990, wall1962, bromic1989,
cabeng1994, alperi1986, konzim1998, kulsha1985, geteth2007, miyach2001}

\newpage
\bibliographystyle{plain}
\bibliography{biblio}

\end{document}